\documentclass[a4paper,reqno,onecolumn,accepted=2024-04-10]{quantumarticle}
\pdfoutput=1
\usepackage{fullpage}
\usepackage[usenames,dvipsnames,svgnames,table]{xcolor}
\usepackage[english]{babel}
\usepackage[hidelinks]{hyperref}
\usepackage{amsmath,amssymb,amsthm,blkarray}
\usepackage{bm} 
\usepackage{mathtools}
\usepackage{graphicx}
\usepackage{enumitem}
\usepackage{ytableau}
\usepackage[capitalize]{cleveref}
\usepackage{tikz,pgfplots}
\pgfplotsset{compat=1.10}
\usepackage{thmtools}
\usepackage{thm-restate}
\usepackage{bm}

\newcommand{\T}{^{\sf T}}

\DeclareMathOperator{\sdp}{sdp}

\DeclareMathOperator{\sign}{sign}
\DeclareMathOperator{\Span}{span}

\newcommand{\ignore}[1]{}
\newcommand{\N}{\mathbb{N}}

\newcommand{\Z}{\mathbb{Z}}

\newcommand{\R}{\mathbb{R}}

\newcommand{\C}{\mathbb{C}}

\newcommand{\Imub}{\mathcal I_{\mathrm{MUB}}}
\newcommand{\id}{\mathbf{\mathrm{id}}}

\let\OLDthebibliography\thebibliography
\renewcommand\thebibliography[1]{
  \OLDthebibliography{#1}
  \setlength{\parskip}{0pt}
  \setlength{\itemsep}{0pt plus 0.3ex}
}

\renewcommand{\sup}{\mathrm{sup}}
\renewcommand{\min}{\mathrm{min}}

\renewcommand{\max}{\mathrm{max}}

\renewcommand{\epsilon}{\varepsilon}

\def\01{\{0,1\}}

\renewcommand{\mathbf}{\bm}

\newtheorem{defin}{Definition}[section]
\newtheorem{definition}[defin]{Definition}
\newtheorem{proposition}[defin]{Proposition}
\newtheorem{theorem}[defin]{Theorem}
\newtheorem{corollary}[defin]{Corollary}
\newtheorem{lemma}[defin]{Lemma}

\newtheorem*{claim*}{Claim}

\newtheorem*{conjecture*}{Conjecture}
\newtheorem{remark}[defin]{Remark}
\newtheorem{examplex}[defin]{Example} 
\newenvironment{example}
  {\pushQED{\qed}\examplex}
  {\popQED\endexamplex}

\newcommand{\ncx}{\langle {\bf x}\rangle}

\newcommand{\MM}{{\mathcal M}}
\newcommand{\MI}{{\mathcal I}}

\newcommand{\A}{\mathcal A}

\begin{document}

\title{Mutually unbiased bases: polynomial optimization and symmetry}
\author{Sander Gribling}
\affiliation{Department of Econometrics and Operations Research, Tilburg University, Tilburg, The
Netherlands}
\email{s.j.gribling@tilburguniversity.edu}
\orcid{0000-0002-6817-2971}

\author{Sven Polak}
\affiliation{Department of Econometrics and Operations Research, Tilburg University, Tilburg, The Netherlands}
\email{s.c.polak@tilburguniversity.edu}
\orcid{0000-0002-4287-6479}

\date{}

\begin{abstract}
A set of $k$ orthonormal bases of $\C^d$ is called mutually unbiased if $|\langle e,f\rangle |^2 = 1/d$ whenever $e$ and $f$ are basis vectors in distinct bases. A natural question is for which pairs $(d,k)$ there exist~$k$ mutually unbiased bases in dimension $d$. 
The (well-known) upper bound $k \leq d+1$ is attained when~$d$ is a power of a prime. For all other dimensions it is an open problem whether the bound can be attained. Navascu\'es, Pironio, and Ac\'in showed how to reformulate the existence question in terms of the existence of a certain $C^*$-algebra. This naturally leads to a noncommutative polynomial optimization problem and an associated hierarchy of semidefinite programs. The problem has a symmetry coming from the wreath product of $S_d$ and~$S_k$. 

We exploit this symmetry (analytically) to reduce the size of the semidefinite programs making them (numerically) tractable. A key step is a novel explicit decomposition of the $S_d \wr S_k$-module $\C^{([d]\times [k])^t}$ into irreducible modules. We present numerical results for small $d,k$ and low levels of the hierarchy. In particular, we obtain sum-of-squares proofs for the (well-known) fact that there do not exist $d+2$ mutually unbiased bases in dimensions~$d=2,3,4,5,6,7,8$. Moreover, our numerical results indicate that a sum-of-squares refutation, in the above-mentioned framework, of the existence of more than $3$ MUBs in dimension $6$ requires polynomials of total degree at least~$12$.  
\end{abstract}

\maketitle

\section{Introduction} 

Let $d \in \N_{\geq 2}$. A set of $k$ orthonormal bases of $\C^d$ is called \emph{mutually unbiased} if 
\begin{equation}\tag{$*$} \label{eq: MUB property}
|\langle e,f\rangle |^2 = \frac{1}{d}
\end{equation}
whenever $e$ and $f$ are basis vectors in \emph{distinct} bases. A natural question is for which pairs $(d,k)$ there exist $k$ mutually unbiased bases in dimension $d$. A dimension counting argument shows that there can be at most $d+1$ mutually unbiased bases (MUBs) in dimension $d$.\footnote{To see this, consider the map $u \mapsto uu^* - I_d/d$ between unit vectors in $\C^d$ and traceless Hermitian $d \times d$ matrices. This map sends $k$ MUBs to $k$ pairwise-orthogonal $(d-1)$-dimensional subspaces of the $(d^2-1)$-dimensional space of traceless Hermitian $d \times d$ matrices. Hence $k \leq \frac{d^2-1}{d-1}=d+1$.}
No other general upper bounds are known. When $d$ is a power of a prime number it is known that there exist $d+1$ MUBs~\cite{Ivanovic81,WF89}. However, for all $d$ that are not a power of a prime the question is wide open. Even for the first such number, dimension~$6$, it is not known whether there exist more than~$3$ such bases, despite extensive numerical search~\cite{BW09}. It is known however that certain sets of 3 MUBs are \emph{maximal}, i.e., they can not be extended, see~\cite{Gra04}. Zauner conjectured that there do not exist~$4$ MUBs in dimension~$6$~\cite{Zauner}. Recently, this (widely believed) conjecture was highlighted as one of five important open problems in quantum information theory \cite{HRLZ22}.

We mention some other constructions of sets of MUBs. If $k$ MUBs exist in dimensions $d_1$ and $d_2$, then $k$ MUBs exist in dimension $d_1 d_2$. Hence, if $d = p_1^{n_1} p_2^{n_2} \cdots p_k^{n_k}$ for prime numbers $p_1,\ldots, p_k$ and integers $n_1,\ldots,n_k$ with $p_1^{n_1} < p_2^{n_2} < \ldots < p_k^{n_k}$, then there exist at least $p_1^{n_1}+1$ MUBs in dimension $d$. This strategy of ``reducing to prime-powers'' does not always lead to the largest number of MUBs in a given non-prime power dimension. An alternative construction was proposed in~\cite{WocjanBeth05}: $k$ mutually orthogonal Latin squares of order $n$ can be used to construct $k+2$ MUBs in dimension $n^2$. For dimension $d = 26^2 =2^2 13^2$ this yields $6$ MUBs (instead of $2^2+1$). Finally, for any dimension $d$, if there exist $d$ MUBs, then there exist $d+1$ MUBs~\cite{Weiner13}. 

\paragraph{Applications of mutually unbiased bases.} A main feature of a collection of MUBs is that if one prepares a state in one of the bases, the measurements of this state with respect to another basis  are predicted to occur uniformly at random. This property gives rise to many applications of MUBs in quantum information theory, for example in tomography algorithms~\cite{WF89}, cryptographic protocols~\cite{bb84}, and entanglement detection~\cite{SHBAH12}. We refer to~\cite{BBELTZ07} for an excellent survey on mutually unbiased bases. Our work is motivated by Zauner's conjecture on the existence of MUBs in dimension $6$. From the perspective of applications, a construction of $7$ MUBs in dimension $6$ would yield optimal schemes in several settings. A proof of non-existence could help us understand what is special about the number/dimension $6$. In the area of combinatorial designs the same number plays a special role, for instance there does not exist a finite affine plane of order $6$~\cite{Bru49}, nor a pair of mutually orthogonal Latin squares of order $6$~\cite{Tar01}. We refer to, e.g., \cite{Woo06} for a discussion about the connections between MUBs and certain combinatorial designs.

\paragraph{$C^*$-algebraic formulation.}

Navascu\'es, Pironio, and Ac\'in~\cite{NPA12} gave a $C^*$-algebraic formulation of the above problem: $k$ MUBs exist in dimension $d$ if and only if a certain $C^*$-algebra has a representation with $I \neq 0$. We describe this formulation in detail in \cref{sec: algebra}. It can be viewed from the perspective of tracial noncommutative polynomial optimization and this leads to semidefinite programming relaxations.
Proving infeasibility of such a relaxation would prove non-existence of $k$ MUBs in dimension $d$. Here we follow this approach and exploit the symmetries in the resulting semidefinite programs. 

\paragraph{Symmetry in semidefinite programming.}
Symmetry is widely used in semidefinite programming and polynomial optimization~\cite{BGSV12}. The usage of symmetry goes back to Delsarte~\cite{D73}, who formulated a linear programming bound for codes which was later shown~\cite{S79} to be a symmetry-reduced version of a slight strengthening of Lov\'asz' $\vartheta$-function. Symmetry in semidefinite programs can be used to reduce their size by block-diagonalizing the involved matrix $*$-algebras. Indeed, it is a consequence of Schur's lemma that a complex matrix $*$-algebra $\mathcal A$ is $*$-isomorphic to a direct sum of full matrix algebras. That is, there exists a $*$-isomorphism $\phi$ such that
\begin{equation} \label{eq: def phi}
\phi(\mathcal A) =\bigoplus_{i=1}^k \C^{m_i \times m_i}.
\end{equation}
Constructing this $*$-isomorphism is often challenging. For numerical methods to do so, see~\cite{MKKK10,KDP11,RMB21}. Via the regular $*$-representation one can obtain a representation of $\mathcal A$ in an (often) smaller matrix algebra~\cite{KPS07}. We focus on the setting where the algebra consists of matrices that are invariant under a permutation action of a (finite) group. That is, we consider the matrix $*$-algebra $\mathcal A = (\C^{Z \times Z})^G$ of $G$-invariant $Z \times Z$ matrices where the group $G$ acts on the set $Z$. In this setting one can construct the $*$-isomorphism from \eqref{eq: def phi} explicitly. To do so, one needs to decompose the $G$-module $\C^Z$ into a direct sum of irreducible $G$-modules (see \cref{sec:symmetry block} for the general theory). Group invariance has been used previously in (commutative) 
polynomial optimization, see for example~\cite{GP04,RTAL13}.  More generally, group symmetry in semidefinite programs has applications in many areas, for example in coding theory~\cite{S05}, combinatorics~\cite{KPS07, KPS09, KS10}, and in geometry~\cite{BV08}.
Finding the $*$-isomorphism (i.e., the decomposition of $\C^Z$) remains challenging even in the group-invariance setting.

Below we first discuss the symmetry that is present in the problem of determining the largest number of MUBs in a given dimension (the MUB problem), and then we give an overview of some of the prior work and our results. 

\paragraph{The symmetry in the MUB problem.} For $n \in \N$, let $S_n$ be the symmetric group on $n$ elements. The MUB property \eqref{eq: MUB property} of a set of $k$ MUBs in dimension~$d$ is naturally invariant under relabeling the bases (an $S_k$ action) and, for each basis, relabeling the basis elements ($S_d$ actions). These actions together precisely define an action of the wreath product $S_d \wr S_k$. One can show that the $C^*$-algebraic formulation and the related semidefinite programming (SDP) relaxations are invariant under this action. Since SDPs are convex, we  may thus restrict our attention to solutions that are invariant under the $S_d \wr S_k$ action. By reducing the number of variables, it yields a first reduction in the complexity of the semidefinite programs. It also implies that the matrix variables in the semidefinite program are group-invariant; by \eqref{eq: def phi} this allows a second reduction in the complexity of the semidefinite programs. Constructing the $*$-isomorphism $\phi$ analytically is the main contribution of our work, see \cref{sec:reprset wreath}.

\paragraph{Symmetry reductions involving the wreath product.} 

In this paper we exploit group symmetry where the group is a wreath product of symmetric groups. Elements of the wreath product~$S_d \wr S_k$ are of the form $(\bm{\alpha},\beta) = ((\alpha_1,\ldots,\alpha_k), \beta)$ where $\alpha_i \in S_d$ for each $i \in [k]$ and $\beta \in S_k$. There are two canonical actions of the wreath product of~$S_d \wr S_k$ (cf., e.g.,~\cite{CST13}). The first action is an action on~$[d]^k$ defined as follows:
\[
((\alpha_1,\ldots,\alpha_k), \beta) \cdot (z_i)_{i \in [k]} = (\alpha_{\beta^{-1} (i)}\cdot z_{\beta^{-1}(i)} )_{i \in [k]}\quad \text{ for $(z_i)_{i \in [k]} \in [d]^k$}. 
\]
This action is also called the primitive action of~$S_d \wr S_k$ on~$[d]^k$ and is used extensively in coding theory related SDPs: see the fundamental work of Schrijver~\cite{S05} and subsequent works that rely on representation theory~\cite{Gij09,LPS17}. The symmetry reduction for the primitive action is by now well understood, see for example \cite{P19} for a general exposition of reducing semidefinite programs invariant under~$G \wr S_n $ acting on~$V^{\otimes n}$ (where~$G$ acts on~$V$).
In this paper we study a second canonical action of~$S_d \wr S_k$, which is an action on $[d] \times [k]$. It is defined as follows:
\[
((\alpha_1,\ldots,\alpha_k), \beta) \cdot (i,j) =  (\alpha_{\beta(j)}(i),\beta(j))  \quad \text{ for $(i,j) \in [d]\times[k]$}. 
\]
This action is known as the imprimitive action of~$S_d \wr S_k$ on~$[d] \times [k]$.
We give an analytic decomposition of $\C^{([d] \times [k])^t}$ into~$S_d \wr S_k$-irreducible submodules (for fixed~$t$), which is a new result for $t>1$. For $t=1$ the separate actions of $S_d$ on~$\C^{[d]^t}$ and $S_k$ on~$\C^{[k]^t}$ are transitive and the analytic decomposition can be obtained for example from~\cite[Thrm.~3.1.1]{CST13}. 

\paragraph{Representation theory of the wreath product $S_d \wr S_k$.}
It is classical theory to describe the irreducible modules for wreath products of finite groups, see for example~\cite{Ker71, MacDonald80}. In particular, the irreducible modules for the wreath product $S_d \wr S_k$ are known. However, as mentioned before, decomposing an arbitrary wreath-product-module into its irreducible components often remains challenging. Recently, certain ``permutation modules'' for wreath products have been studied (see \cref{sec: iso to perm module} for their definition), and methods to derive the multiplicities of the irreducible modules in these modules have been given \cite{CT03, Green19}. However, in general there is no explicit description of the homomorphisms from the irreducible modules to the permutation modules (which is what we need for our symmetry reduction), see~\cite{Green19} for a partial description. Our symmetry reduction involves constructing explicit homomorphisms for ``L-shaped'' permutation modules. 

\paragraph{Related prior work.} In~\cite{BW10} the MUB problem is considered from the perspective of solving a system of polynomial equations in commutative variables. The authors show that both techniques based on Gr\"obner bases and techniques based on polynomial optimization can be used to rule out the existence of certain MUBs in small dimensions. An advantage of this approach is its flexibility: one can for example fix a set of MUBs and ask whether it can be extended to a larger set of MUBs. A disadvantage is the number of variables (one per real variable). Noncommutative polynomial optimization was used in~\cite{NPA12} and also in~\cite{ABMP18}. The latter gives a reformulation of the MUB problem in terms of a game based on quantum random access codes: a certain winning probability is achievable if and only if there exist $k$ MUBs in dimension $d$. They then use an SDP hierarchy of upper bounds on the winning probability due to Navascu\'es and V\'ertesi~\cite{NV15} and numerically exploit the symmetry in the corresponding SDPs. Using this approach they rule out the existence of $d+2$ MUBs in dimensions $d=3,4$. A symmetry reduction for their SDPs can be obtained by decomposing certain modules of the wreath product~$S_d \wr S_k$, just as in our case. However, the moment matrices appearing in their approach have a different index set from our moment matrices, hence our symmetry reduction is not directly applicable. We leave it to future work to combine our analytical symmetry reduction with their approach.

\paragraph{Contributions.} We obtain a novel explicit decomposition of the $S_d \wr S_k$-module $\C^{([d]\times [k])^t}$ into irreducible modules. In particular, this allows us to obtain the analytic symmetry reduction of the SDP relaxations coming from the C$^*$-algebraic formulation of Navascu\'es, Pironio, and Ac\'in~\cite{NPA12}. For fixed $t \in \N$, the size of the symmetry reduced SDP is independent of $d$ and $k$ whenever $d,k \geq 2t$. We provide an implementation of the symmetry reduced SDPs and use it to obtain numerical sum-of-squares certificates for the non-existence of $d+2$ MUBs in dimensions $d=2,3,4,5,6,7,8$. Moreover, our numerical results indicate that a sum-of-squares refutation (in the context of~\cite{NPA12}) of the existence of more than $3$ MUBs in dimension $6$ requires polynomials of total degree at least~$12$. See Table~\ref{table: introtable} for an overview of these numerical results. We use the SDP-solvers SDPA-DD or SDPA-GMP~\cite{Nakata10,YFFKNN12} (when a solver returns an unknown result status we use a higher-precision solver). 
\begin{table}[ht]\centering  \small
    \begin{tabular}{| r | r | r | r| r|r| r| r || l |}
    \hline
    $d$ & $k$ & $t$ &  size=$(dk)^{\lfloor t\rfloor}$&\#vars & \multicolumn{1}{r|}{$\#$linear} &  \multicolumn{2}{c||}{block sizes} & result   \\ \cline{7-8}
     & & & & & constraints & sum & max &   \\\hline
2 & 4 & 4.5 & 4096 & 7 & 8  & 472 & 85 &\textbf{infeasible}  \\
3 & 5 & 4.5 & 50625 & 7 & 2  & 1259 & 142&  \textbf{infeasible}\\
4 & 6 & 5 & 7962624 & 38 & 2  & 6374 & 389 &  \textbf{infeasible}\\
5 & 7 & 5 & 52521875 & 38 & 2  & 6732 & 389 & {{\textbf{infeasible}}}  \\
 6 & 8 & 5 & 254803968 & 38 & 2  & 6820 & 389& \textbf{infeasible} 
\\ 
7 & 9 & 5 & 992436543 & 38 & 2  &6830 & 389 & \textbf{infeasible} 
\\ 
8 & 10 & 5 & 3276800000 & 38 & 2 & 6831 & 389 & \textbf{infeasible} \\ \hline 
6 & 4 & 5.5 & 7962624 & 43 & 3 & 8049 & 577 & feasible 
\\ 6 & 7 & {$5.5$} &  130691232& 62 & 3 & 18538 & 1107 & feasible  
\\
\hline
    \end{tabular} 
      \caption{For various $d,k,t$, we give the number of rows of the original $t$-th level SDP, and for the symmetry reduced SDP the number of variables, the number of linear constraints on the variables, the sum and max of the block sizes, and the solver status.} \label{table: introtable}
\end{table} 

\paragraph{Implementation.}
We use the programming language Julia to compute the symmetry-reduced semidefinite programs. The code is available here: 
\begin{center}
\url{https://github.com/MutuallyUnbiasedBases/}
\end{center}
Our code roughly consists of two parts. We first compute a representative set (cf.~\cref{prop: symmetry reduction}) for the representation at hand. We do this both for the action of $S_d \wr S_k$ on (homogeneous) noncommutative polynomials in variables $x_{i,j}$ with $i \in [d]$, $j \in [k]$, and for the action of $S_k$ on (homogeneous) noncommutative polynomials in variables $x_1,\ldots,x_k$. This first part does not use the additional structure present in the MUB problem and is therefore also applicable to other problems that exhibit the same symmetry. 
The second part then consists of computing the entries of the symmetry-reduced SDP. For this we do use the specific structure of the MUB problem (and the previously computed representative set), see \cref{sec: L value} for more details.

\subsection{Organization}
 
In \cref{sec:symmetry} we provide preliminaries about (noncommutative) polynomial optimization and the general framework of symmetry reductions of group-invariant semidefinite programs. In \cref{sec: algebra} we state the $C^*$-algebraic formulation of the MUB problem of Navascu\'es, Pironio, and Ac\'in, and we discuss properties of the related semidefinite programming relaxations.  To introduce the well-known representation theory of the symmetric group, we provide in \cref{sec:reprset sk} the symmetry reduction of noncommutative polynomial optimization problems in variables $x_1,\ldots,x_k$ that are invariant under the action of $S_k$. We present this from the perspective of representation theory (\cref{sec: prelim Sk,reprsetSk}), but we also illustrate the theory in \cref{sec: example 32} by presenting a small example (degree-$2$ polynomials in $x_1,x_2,x_3$)) in the language of noncommutative polynomials: the symmetry reduction amounts to choosing a suitable basis of the space of polynomials. 

The symmetry reduction of $S_k$ serves as a first step towards that of $S_d \wr S_k$ in two ways. First, it can be viewed as a special case of the full $S_d \wr S_k$ action on the variables $x_{i,j}$ ($i \in [d], j \in [k]$) by restricting to polynomials in the variables $x_{1,j}$ ($j \in [k]$). We can thus use it to symmetry-reduce a relaxation of the SDPs presented in \cref{sec: SDP} where we restrict to polynomials in $x_{1,j}$ ($j \in [k]$), in \cref{sec: numerics} we will see that these relaxations sometimes already exclude the existence of MUBs. Second, the representation theory of $S_d \wr S_k$ naturally builds on that of the symmetric group. Indeed, in \cref{sec: rep wreath} we describe known (irreducible) $S_d \wr S_k$-modules that can be constructed from (irreducible) modules for the symmetric group. 

Our main result, the analytic symmetry reduction of the $S_d \wr S_k$-invariant semidefinite programs, is obtained in \cref{sec:reprset wreath}. At a high level, we obtain our symmetry reduction by decomposing the $S_d \wr S_k$-module $\C\ncx_{=t}$ given by (homogeneous) degree-$t$ polynomials in the variables $x_{i,j}$ ($i \in [d], j \in [k]$) into irreducible modules. We do so by first decomposing $\C\ncx_{=t}$ into orbits, we then relate these orbits to known modules of $S_d \wr S_k$. It turns out that these modules have a special structure (cf.~\cref{rem: our W is nice}) which allows us to give an explicit decomposition of them into irreducible modules. To illustrate the (rather technical) theory, we give the explicit symmetry-reduction for the case $d=2$, $k=3$, $t=2$ in the language of noncommutative polynomials in \cref{sec: example 232}. In \cref{sec: iso to perm module} we relate the results of \cref{sec:reprset wreath} to the literature: we show they can be viewed as a novel explicit decomposition of an ``L-shaped'' permutation module into irreducible $S_d \wr S_k$-modules. 

Finally, in \cref{sec: numerics} we provide numerical results obtained using our symmetry reduction.

\section{Preliminaries on polynomial optimization \& symmetry} \label{sec:symmetry}

The use of sum-of-squares and moment techniques for polynomial optimization problems goes back to the works of Parrilo and Lasserre~\cite{Par00,Las01}. We refer to, e.g., \cite{Las09,Lau09} for more information about commutative polynomial optimization. In the noncommutative setting such techniques were first introduced for eigenvalue optimization by Navascu\'es, Pironio, and Ac\'in~\cite{NPA10}. The extension to the tracial setting can be found in~\cite{BCKP13,KP16}. The method of symmetry reduction of group-invariant matrix algebras that we describe in \cref{sec:symmetry block} is well known, see for example \cite{Gij09} and \cite{P19} for complete proofs of the statements we use here. 

\subsection{Noncommutative polynomial optimization}
We use $\R\langle x_1,\ldots,x_n\rangle$ to denote the ring of noncommutative polynomials in variables $x_{i}$ for $i \in [n]$. Whenever the number of variables is clear from context, we abbreviate $\R\langle x_1,\ldots,x_n\rangle$ to $\R\ncx$. We equip $\R\ncx$ with an involution $^*$ (a linear map) which acts on a monomial $x_{i_1} x_{i_2} \cdots x_{i_\ell} \in \ncx$ by reversing the order of the variables: $(x_{i_1} x_{i_2} \cdots x_{i_\ell})^* = x_{i_\ell} x_{i_{\ell-1}} \cdots x_{i_1}$. In particular, we have $x_i^* = x_i$ for all $i$. A polynomial $p$ is Hermitian if $p^* = p$. 
Let $S$ be a set of Hermitian noncommutative polynomials. The matrix positivity domain associated to $S$ is  
\[
\mathcal D(S) = \cup_{d \in \N} \{(X_1, \ldots, X_n) \in (\mathrm H^d)^n : g({\mathbf X}) \succeq 0 \text{ for all } g \in S\}
\]
where $\mathrm H^d$ is the space of Hermitian $d \times d$ matrices  
and for a $C^*$-algebra $\A$ we let 
\[
\mathcal D_{\A}(S):=\{(X_1,\ldots, X_n) \in \A^n: X_i^* = X_i \text{ for } i \in [n], \  g({\mathbf X}) \succeq 0 \text{ for all } g \in S\}
\]
be its $C^*$-algebraic analogue. One similarly defines a matrix variety and its $C^*$-algebraic analogue for a set $T$ of noncommutative polynomials: 
\begin{align*}
\mathcal V(T) &= \cup_{d \in \N} \{(X_1, \ldots, X_n) \in (\mathrm H^d)^n : h({\mathbf X}) = 0 \text{ for all } h \in T\} \\
\mathcal V_{\A}(T) &= \{(X_1, \ldots, X_n) \in \A^n : X_i^* = X_i \text{ for } i \in [n], \ h({\mathbf X}) = 0 \text{ for all } h \in T\}.
\end{align*}
A tracial polynomial optimization problem over $\mathcal D_{\A}(S) \cap \mathcal V_\A(T)$ can be written as 
\begin{equation} \label{eq: tracial pop}
    \inf_{\A,\tau, \mathbf X}  \{ \tau(f(\mathbf X))  :\ \tau \text{ is a tracial state on } \A,\ \mathbf X \in\mathcal D_{\A}(S) \cap \mathcal V_\A(T).\}
\end{equation}
A fundamental result in the theory of noncommutative polynomial optimization is a reformulation of the above in terms of linear functionals acting on the vector space of polynomials in noncommutative variables $x_1,\ldots, x_n$. That is, one considers the following (related) problem
\begin{equation}
    \inf \{\, L(f) :\ L\in \R\ncx^* \text{ tracial},\ L(1) = 1, \ L \geq 0 \text{ on }\MM(S), \ L=0 \text{ on }\MI(T)\, \}. \label{eq:full}
\end{equation}
Here $\MM(S)$ is the \emph{quadratic module} associated to the set $S$, i.e., the cone of polynomials generated by polynomials of the form $p^* g p$ for $g \in S$, $p \in \R\ncx$, and $\MI(T)$ is the \emph{ideal} associated to $T$, i.e., the vector space of polynomials of the form $p h$ for $h \in T$, $p \in \R\ncx$. The problems \eqref{eq: tracial pop} and \eqref{eq:full} are equivalent whenever the Minkowski sum $\MM(S)+\MI(T)$ is \emph{Archimedean}\footnote{$\MM(S)+\MI(T)$ is Archimedean whenever there exists an $R \in \R$ such that $R^2- \sum_{i \in [n]} x_i^* x_i \in \MM(S)+\MI(T)$. This property will hold for the polynomial optimization problems we consider in this work.}. Proofs of this equivalence can be found in \cite{NPA12,BCKP13}. One obtains a hierarchy of lower bounds by considering polynomials of degree at most $2t$ (for $t \in \N$):
\begin{equation}\label{eq:ft}
 f_t = \inf \{\, L(f) :\ L\in \R\ncx^*_{2t} \text{ tracial},\ L(1) = 1, \ L \geq 0 \text{ on }\MM(S)_{2t}, \ L=0 \text{ on }\MI(T)_{2t}\, \}, 
\end{equation}
where 
\begin{align*}
\MM(S)_{2t} &= \mathrm{cone}\{p^* gp : g \in S, p \in \R\ncx, \ \mathrm{deg}(p^*gp) \leq 2t\}, \\
\MI(T)_{2t} &= \mathrm{cone}\{p h : h \in T, p \in \R\ncx, \ \mathrm{deg}(ph) \leq 2t\}.
\end{align*}
This hierarchy converges to \cref{eq:full} whenever $\MM(S)+\MI(T)$ is Archimedean. Each relaxation~$f_t$ can be expressed as a semidefinite program (cf.~\cref{eq: positivity vs moment matrix}). 

\subsection{Group-invariant polynomial optimization problems.} \label{sec: group invariant pop}
Let $G$ be a finite group acting on $[n]$. We let $G$ act on noncommutative variables $x_1,\ldots, x_n$ by letting $G$ act on the indices: $\sigma \cdot x_i = x_{\sigma(i)}$ for all $\sigma \in G, i \in [n]$. We extend the action of $G$ to monomials by letting $\sigma$ act on each of the variables, e.g., $\sigma \cdot (x_i x_j) = x_{\sigma(i)} x_{\sigma(j)}$, and we further extend it linearly to $\R\ncx = \R\langle x_1,\ldots,x_n\rangle$. 
 We define its action on $n$-tuples $(x_1,\ldots, x_n) \in \A^n$ similarly: $\sigma \cdot (x_1,\ldots,x_n) = (x_{\sigma(1)},\ldots, x_{\sigma(n)})$. For $\sigma \in G$ and $L \in \R\ncx^*$ we define $\sigma \cdot L$ via $\sigma \cdot L(p) = L(\sigma \cdot p)$ for all $p \in \R\ncx$, this defines an action of $G$ on $\R\ncx^*$. We say $L \in \R\ncx^*$ is \emph{$G$-invariant} if $\sigma \cdot L = L$, i.e., if $L(p) = L(\sigma \cdot p)$ for all $\sigma \in G$ and $p \in \R\ncx$. Let $\R\ncx^G$ be the subspace of $G$-invariant polynomials.

In the polynomial optimization problems that we consider, there will be a finite group $G$ such that the objective $f$ is invariant under the action of $G$ (i.e., $f \in \R\ncx^G$) and such that whenever $L$ is feasible for the moment relaxation \eqref{eq:ft}, then also $\sigma \cdot L$ is feasible for every $\sigma \in G$. Note that $\sigma \cdot L$ will have the same objective value $\sigma \cdot L(f) = L(f)$ since $f$ is $G$-invariant.
As the feasible region in~$\eqref{eq:ft}$ is convex and~$G$ is finite, it is easy to verify that the linear functional $\overline L = \tfrac{1}{|G|} \sum_{\sigma \in G} \sigma \cdot L$ is a $G$-invariant feasible solution with the same value. Hence we may assume that the optimum~$L$ in~\eqref{eq:ft} is~$G$-invariant.
Note that $G$-invariant linear functionals are completely determined by their values on $\R\ncx^G$. This gives a first reduction in the size of the SDP by reducing the number of variables. As mentioned before, a second reduction can be obtained by exploiting the group invariance of the algebra generated by the matrix variables in the SDP. In the next section we recall the general theory. 

\subsection{Symmetry reduction of group-invariant matrix algebras} \label{sec:symmetry block}

Let $G$ be a finite group. If $G$ acts  on a finite-dimensional complex vector space $V$, we call $V$ a $G$-module.\footnote{A \emph{group action} of a group $G$ on a set~$Z$ is a map $G \times Z \to Z$: $(g,z) \mapsto g \cdot z$ which satisfies $1 \cdot z = z$ and~$g_1 \cdot (g_2 \cdot z) = (g_1g_2)\cdot z$ for all~$g_1,g_2 \in G$ and~$z\in Z$. When~$Z$ is a (complex) linear space, we additionally assume~$g\cdot (z_1+z_2) = g\cdot z_1 + g \cdot z_2$ and~$g\cdot (\alpha z)= \alpha (g \cdot z) $ for each~$z_1,z_2 \in Z$, $g \in G$ and~$\alpha \in \C$.} Suppose that~$V$ is equipped with a $G$-invariant inner product $\langle \cdot,\cdot \rangle$. Let $\mathrm{End}_G(V)$ be the space of $G$-invariant endomorphisms on $V$, that is, the space of linear maps~$A\,: \,V \to V$ that are $G$-invariant: $A (g \cdot v) =  g \cdot (A  v)$ for all $v \in V$ and $g \in G$. The space $\mathrm{End}_G(V)$ forms an algebra with multiplication given by function composition. We recall how to block-diagonalize~$\mathrm{End}_G(V)$. 

If $U \subseteq V$ is a $G$-invariant subspace of a $G$-module $V$, then $U$ is a \emph{submodule} of $V$. A module $V \neq \{0\}$ is irreducible if the only submodules of $V$ are $\{0\}$ and $V$. Two $G$-modules $V$ and $W$ are \emph{$G$-isomorphic} if there exists a bijection $\psi:V \to W$ that respects the group action $g \cdot \psi(v) = \psi(g \cdot v)$ for all $v \in V$, $g \in G$. A $G$-module $V$ can then be decomposed as 
\begin{equation} \label{eq: irreducible decomp}
V = \bigoplus_{i=1}^k \bigoplus_{j=1}^{m_i} V_{i,j}
\end{equation}
where the $V_{i,j}$ are irreducible $G$-modules and $V_{i,j}$ and $V_{i',j'}$ are $G$-isomorphic if and only if $i'=i$. One can show that $\mathrm{End}_G(V)$ is $*$-isomorphic to the matrix algebra $\oplus_{i=1}^k \C^{m_i \times m_i}$ (a $*$-isomorphism between two matrix algebras is an isomorphism that preserves the $*$-operation, i.e., the conjugate transpose). Here we will use a slightly weaker result: there exists a linear bijection $\Phi:\mathrm{End}_G(V) \to \oplus_{i=1}^k \C^{m_i \times m_i}$ such that for all $A \in \mathrm{End}_G(V)$,  $A \succeq 0$ if and only if $\Phi(A) \succeq 0$. (In general the $\Phi$ we construct will not be a $*$-isomorphism.)

\begin{definition}[Representative set] \label{def: representative set}
Let $V$ be a $G$-module with a decomposition into irreducible $G$-modules $V_{i,j}$ as in~\eqref{eq: irreducible decomp}. For each $i \in [k]$, choose vectors $u_{i,j} \in V_{i,j}$ for $j \in [m_i]$ in such a way that for each $j, \ell \in [m_i]$ there exists a $G$-isomorphism from $V_{i,j}$ to $V_{i,\ell}$ that maps $u_{i,j}$ to $u_{i,\ell}$. Let $U_i = (u_{i,1},\ldots,u_{i,m_i})$. Then $\{U_1,\ldots, U_k\}$ forms a \emph{representative set} for the action of $G$ on $V$. 
\end{definition}

The term \emph{symmetry adapted basis} is sometimes used to denote a basis respecting the decomposition of~$V$ into~$G$-modules, cf.~\cite{FS92}. A representative set contains one `representative' basis element of each irreducible module, cf.~\cite{LPS17}. The term \emph{symmetry basis} is also used in the literature to denote the set of consisting of all~$u_{i,j}$ forming a representative set, cf.~\cite{BR21}.

We then have the following (which is well-known general theory; a proof of the theorem in this exact form can be found in~\cite{P19}):
\begin{proposition} \label{prop: symmetry reduction}
Let $G$ be a finite group and let $V$ be a finite-dimensional $G$-module with a decomposition as in~\eqref{eq: irreducible decomp}. Let $\{U_1,\ldots, U_k\}$ be a representative set for the action of $G$ on $V$. Define the map $\Phi:\mathrm{End}_G(V) \to \bigoplus_{i=1}^k \C^{m_i \times m_i}$ as 
\begin{equation*} 
\begin{split}
A &\mapsto \bigoplus_{i=1}^k \big(\langle u_{i,j}, A u_{i,\ell}\rangle\big)_{j,\ell \in [m_i]}.
\end{split}
\end{equation*}
Then $\Phi$ is a bijection and for all $A \in \mathrm{End}_G(V)$,  we have $A \succeq 0 \text{ if and only if } \Phi(A) \succeq 0$.
\end{proposition}
The complex vector spaces~$V$ studied in this paper all have the form~$V=\C^Z$, where~$Z$ is a finite set on which~$G$ acts (hence~$G$ acts on~$V=\C^Z$ via permutation matrices). In this case $\mathrm{End}_G(V)$ is naturally isomorphic to~$(\C^{Z \times Z})^G$, the space of complex~$Z \times Z$-matrices which are invariant under the simultaneous action of~$G$ on their rows and columns. Any representative set $\{U_1,\ldots,U_k\}$ as in Definition~\ref{def: representative set} then consists of $m_i$-tuples~$U_i$ of vectors contained in in~$\C^Z$, so we can view~$U_i$ as a $(Z \times m_i)$-matrix. Moreover, each representative set~$\{U_1,\ldots,U_k\}$ determined in this paper consists of real matrices. In that case  $\Phi$ is even a block-diagonalization ``over $\R$'', i.e., 
$\Phi$ provides a bijection between $(\R^{Z \times Z})^G$ and $\bigoplus_{i=1}^k \R^{m_i \times m_i}$ such that $A \succeq 0$ if and only if $\Phi(A) \succeq 0$. 

Finally, note that if~$\{U^1_1,\ldots,U^1_{k_1}\}$ and~$\{U^2_1,\ldots,U^2_{k_2}\}$ are representative sets for the actions of~$G_1$ and~$G_2$ on~$V_1$ and~$V_2$ respectively, then
\begin{align} \label{eq: prodrep}
\{U^1_i \otimes U^2_j \, \, | \,\, i \in [k_1], j \in [k_2]  \}
\end{align} 
is a representative set for the action of~$G_1 \times G_2$ on~$V_1 \otimes V_2$. 

\section{MUB-algebra} \label{sec: algebra}
One can consider an orthonormal basis of $\C^d$ as being defined by vectors $u_1,\ldots, u_d \in \C^d$ or by the associated rank-1 orthogonal projectors $P_1,\ldots, P_d \in \C^{d \times d}$ (where $P_i = u_i u_i^*$ for all $i$). Here we take the second perspective and consider the algebra generated by the projectors arising from (pairwise) mutually unbiased bases. Let $\{u_1,\ldots, u_d\}$ and $\{v_1,\ldots,v_d\}$ form two orthonormal bases of $\C^d$ and set $P_i = u_i u_i^*$, $Q_i = v_i v_i^*$ for all $i \in [d]$. Then the following holds: 
\begin{enumerate}[label=(\roman*)]
\item $P_i^* = P_i = P_i^2$ for all $i \in [d]$, since 
$(u_i u_i^*)^*= u_i u_i^* = \langle u_i, u_i \rangle u_i u_i^* = (u_i u_i^*)^2$,
\item $P_i P_j = \delta_{i,j} P_i \text{ for all } i,j \in [d]$, since 
$u_i u_i^* u_j u_j^* = \langle u_i, u_j \rangle u_i u_j^* = \delta_{i,j} u_i u_i^*$,
\item $\sum_{i=1}^d P_i = I$, since for any $j \in [d]$ we have $
(\sum_{i=1}^d u_i u_i^*) u_j = \sum_{i \in [d]} \langle u_i, u_j \rangle u_i = u_j$,
\item $P_i Q_j P_i = \frac{1}{d} P_i \text{ for all } i,j \in [d]$, since 
$u_i u_i^* v_j v_j^* u_i u_i^* = |\langle u_i, v_j\rangle|^2 u_i u_i^* = \frac{1}{d} u_i u_i^*$,
\item $[P_i U P_i, P_i V P_i] = P_i U P_i V P_i - P_i V P_i U P_i = 0 \text{ for all } i \in [d], U,V \in \C^{d \times d}$, since 
\[
u_i u_i^* U u_i u_i^* V u_i u_i^* = \langle u_i, U u_i\rangle \langle u_i, V u_i \rangle u_i u_i^* = u_i u_i^* V u_i u_i^* U u_i u_i^*.
\]
\end{enumerate}
The above conditions generalize naturally to $k$ pairwise mutually unbiased bases. Note however that the $P_i$'s and $Q_j$'s are $d \times d$ matrices. In order to use noncommutative polynomial optimization techniques, we need to relax the condition that the operators are of size $d \times d$. We require the following $C^*$-algebraic formulation:

\begin{definition}[MUB-algebra] \label{def:mubalg}
We call a $C^*$-algebra $\mathcal A$ a \emph{$(d,k)$-MUB-algebra} if it contains Hermitian elements $X_{i,j}$ for $i \in [d], j \in [k]$ that satisfy the following conditions:
\begin{enumerate}[label=(\roman*)]
    \item $X_{i,j} X_{i',j} = \delta_{i,i'} X_{i,j}$ for all $i,i' \in [d], j \in [k]$, \label{item:1}
    \item $\sum_{i \in [d]} X_{i,j} = I$ for all $j \in [k]$, \label{item:2}
    \item $X_{i,j} X_{i',j'} X_{i,j} = \frac{1}{d} X_{i,j}$ for all $i,i' \in [d], j,j' \in [k]$ with $j \neq j'$,\label{item:3}
    \item $[X_{i,j} U X_{i,j}, X_{i,j} V X_{i,j}] = 0$ for all $i \in [d],j \in [k]$ and $U,V \in \langle \mathbf X\rangle$, where~$\langle \mathbf X \rangle$ denotes the set of monomials in~$X_{i,j}$.\footnote{The proof of \cref{thrmC*alg} only uses $U,V \in \langle \mathbf X \rangle_3$, i.e., monomials in $X_{i,j}$ of degree~$\leq 3$. \label{footnote degree}} \label{item:4}
\end{enumerate}
\end{definition}
We have shown above that if there exist $k$ mutually unbiased bases in dimension $d$, then there exists a $(d,k)$-MUB-algebra. Navascu\'es, Pironio, and Ac\'in showed the converse in~\cite{NPA12}: if a $(d,k)$-MUB-algebra exists with $I \neq 0$, then there exist $k$ mutually unbiased bases in dimension $d$. 
\begin{restatable}[\cite{NPA12}]{theorem}{thrmCalg} \label{thrmC*alg}
There exists a set of $k$ MUBs in dimension $d$ if and only if there exists a $(d,k)$-MUB-algebra with $I \neq 0$.
\end{restatable}
We provide a slightly simpler proof (which is inspired by the proof in \cite{NPA12}) of the ``only if''-direction in \cref{appA}. In \cite{NPA12} the authors suggest to use noncommutative polynomial optimization techniques to determine the existence of a $(d,k)$-MUB-algebra. We use \emph{tracial} noncommutative polynomial optimization to do so, that is, we use the following simple corollary. 
\begin{corollary}
There exists a set of $k$ MUBs in dimension $d$ if and only if there exists a $(d,k)$-MUB-algebra equipped with a tracial state $\tau$ for which $\tau(X_{i,j})=1$ for all $i \in [d], j \in [k]$. 
\end{corollary}
In other words, there exist $k$ MUBs in dimension $d$ if and only if the following problem is feasible 
\begin{align*} 
\inf\{0:&\, \exists\, C^*\text{-alg.~}\A \text{ with tracial state } \tau, \text{ with Herm.~} X_{i,j} \in \A \ (\text{for } i\in[d], j \in [k]) \\
    &\text{ satisfying \ref{item:1}--\ref{item:4} and }\tau(X_{i,j})=1 \},
\end{align*}
where for any set~$\mathcal{F}$, feasibility of the optimization problem $\text{inf}\{0 : \exists\, a \in \mathcal{F} \}$ means that $\mathcal{F}$ is nonempty. 
We now rephrase the above feasibility problem as a tracial optimization problem in the variables~$x_{i,j}$ where $i \in [d], j \in [k]$. Let $\Imub$ be the ideal generated by the relations \ref{item:1} to \ref{item:4} of \cref{def:mubalg}. Then the above feasibility problem can be equivalently formulated as 
\begin{align*}
\inf\{0:&\, \exists\, L \in \R\ncx^* \text{ s.t. } L \text{ is tracial, positive, } \\
&L=0 \text{ on } \Imub, \text{ and } L(x_{i,j})=1 \text{ for all } i \in [d], j \in [k]\}.
\end{align*}
(To see the equivalence, a tracial state $\tau$ defines a feasible solution $L$ by restricting to the algebra generated by the $x_{i,j}$. Conversely, applying the GNS construction to a feasible solution $L$ yields a $C^*$-algebra $\mathcal A$ with the desired properties.)
\begin{theorem} \label{thrm:linearfunc}
There exist $k$ MUBs in dimension $d$ if and only if there exists a positive tracial linear functional $L \in \R\ncx^*$ that satisfies $L=0$ on $\Imub$ and $L(x_{i,j}) = 1$ for all $i \in [d], j \in [k]$. 
\end{theorem}

\subsection{Semidefinite programming relaxation} \label{sec: SDP}

\cref{thrm:linearfunc} gives rise to an infinite-dimensional semidefinite program since positivity of a linear functional $L \in \R\ncx^*$ can be expressed in terms of positive semidefiniteness of its moment matrix:
\begin{equation} \label{eq: positivity vs moment matrix}
L(p^*p) \geq 0 \text{ for all } p \in \R\ncx \quad \Longleftrightarrow \quad M(L) := (L(u^*v))_{u,v \in \ncx} \succeq 0.
\end{equation}
We then obtain (finite-dimensional) relaxations by considering linear functionals acting on polynomials of bounded degree. We define the following semidefinite programs for $t \in \N$:
\begin{align*}
\sdp(d,k,t) =     \inf\{0:&\, \exists\, L \in \R\ncx_{2t}^* \text{ s.t. } L \text{ is tracial},\notag  \\ 
&L=0 \text{ on } {\Imub}_{,2t},\notag  \\  
&L(p^*p) \geq 0 \text{ for all } p \in \R\ncx_t, \notag\\ 
&L(x_{i,j})=1 \text{ for all } i \in [d], j \in [k]\}. 
\end{align*}
Also, we define $\sdp(d,k,t+\frac{1}{2})$ by restricting in the definition of $\sdp(d,k,t+1)$ to polynomials of degree $2t+1$ and for the positivity condition to polynomials of the form $p = x_{i,j} \tilde p$ where~$\tilde p \in \R \langle x \rangle_t$ (where we view $\tilde p^* x_{i,j}^2 \tilde p = \tilde p^* x_{i,j} \tilde p$ as a polynomial of degree $2t+1$). That is, for $t \in \N$, we define 
\begin{align*}
\sdp(d,k,t+\tfrac 12 ) =     \inf\{0:&\, \exists\, L \in \R\ncx_{2t+1}^* \text{ s.t. } L \text{ is tracial},\notag  \\ 
&L=0 \text{ on } {\Imub}_{,2t+1},\notag  \\  
&L(p^*x_{i,j} p) \geq 0 \text{ for all } p \in \R\ncx_t, i\in [d], j \in [k], \notag\\ 
&L(x_{i,j})=1 \text{ for all } i \in [d], j \in [k]\}. 
\end{align*}
These relaxations become tighter as $t$ grows. Indeed, if $L \in \R\ncx_{2t+2}^*$ is feasible for $\sdp(d,k,t+1)$, then the restriction of $L$ to degree $2t+1$ polynomials is feasible for $\sdp(d,k,t+\frac12)$  (using $x_{i,j}^2=x_{i,j}$). Similarly, if $L \in \R\ncx_{2t+1}^*$ is feasible for $\sdp(d,k,t+\frac12)$, then for any polynomial $p \in \R\ncx_{t}$ and $j \in [k]$ we have $L(p^*p) = \sum_{i \in [d]} L(p^* x_{i,j} p) \geq 0$, and therefore the restriction of $L$ to degree $2t$ polynomials is feasible for $\sdp(d,k,t)$.  
We also have the following: 
\begin{lemma}
Let $t \in \N$. 
Then if $\sdp(d,k,t)$ is infeasible, there do not exist $k$ MUBs in dimension~$d$. 
\end{lemma}
Its converse holds in the following sense (its proof is standard, we include it for completeness):
\begin{lemma}
If there do not exist $k$ MUBs in dimension $d$, then there exists a $t_0 \in \N$ such that $\sdp(d,k,t)$ is infeasible for all $t \geq t_0$.
\end{lemma}
\begin{proof}
Indeed, towards a contradiction, suppose $\sdp(d,k,t)$ is feasible for all $t \in \N$. For each $t \in \N$, let $L_t \in \R\ncx_{2t}^*$ be a feasible solution for $\sdp(d,k,t)$. We can extend $L_t$ to a linear functional $\widetilde{L_t}$ on $\R\ncx$ by defining $\widetilde{L_t}(w) = 0$ for all monomials $w \in \ncx$ with $|w| >2t$. The Archimedeanity of $\Imub$ and the fact that $L_t(1)=d$ for all $t$ together imply that $|\widetilde{L_t}(w)| \leq d R^{|w|}$ for all monomials $w \in \ncx$ and all $t \in \N$, where $R \in \R$ is such that $R^2 - \sum_{i,j} x_{i,j}^2 \in \mathcal M(\emptyset) + \Imub$. The Banach-Alaoglu theorem then ensures the existence of a subsequence of the sequence $\{\widetilde{L_t}\}_{t \in \N}$ that converges pointwise to $L \in \R\ncx^*$. This $L$ is then a positive tracial linear functional and it is zero on $\Imub$. \cref{thrm:linearfunc} then shows that there exist $k$ MUBs in dimension $d$. 
\end{proof}

In $\sdp(d,k,t)$ it suffices to enforce $L(p^*p) \geq 0$ for all \emph{homogeneous} polynomials of degree~$t$. Indeed, for $p = \sum_{w \in \ncx_t} p_w w \in \R\ncx_t$, define its homogenization $p_{\mathrm{hom}} = \sum_{w \in \ncx_t} p_w (\sum_{i=1}^d x_{i,1})^{t-|w|} w$. One then has $L(p^*p) = L(p_{\mathrm{hom}}^*p_{\mathrm{hom}})$ since ${\Imub}_{,2t}$ contains generators of the form $1-\sum_{i=1}^d x_{i,j}$ (and $L$ is tracial).  The analogous statement holds for $\sdp(d,k,t+\frac12)$.
\begin{lemma}\label{lemma:degreelemma}
Let $t \in \N$ and let $L \in \R\ncx_{2t}^*$ be a tracial linear functional such that $L=0$ on the truncated ideal generated by the polynomial $1-\sum_{i=1}^d x_{i,1}$. Then, if $L(q^*q)\geq 0$ for all homogeneous polynomials $q$ of degree $t$, then $L(p^*p)\geq 0$ for all polynomials $p$ of degree at most $t$. 
\end{lemma}

\paragraph{Certificates of infeasibility.}

For $d,k,t \in \N$, the dual of $\sdp(d,k,t)$ is
\begin{align*}
    \sdp^*(d,k,t) := \sup \Big\{\sum_{i \in [d],j \in [k]} \lambda_{i,j} :-\sum_{i,j} \lambda_{i,j} x_{i,j} \in \mathcal M(1)_{2t} + \mathcal I_{\mathrm{MUB},2t} \Big\}.
\end{align*}
The program $\sdp^*(d,k,t)$ is strictly feasible. Indeed, the relations $x_{i,j} = x_{i,j}^2$ show that the point $(\lambda_{i,j})_{i \in [d], j \in [k]}$ given by $\lambda_{i,j} = -1+\delta_{i,j}$ is feasible whenever $|\delta_{i,j}|\leq 1$ for all $i \in[d],j\in[k]$. Hence, if $\sdp(d,k,t)$ is infeasible, there exists a feasible solution $(\lambda_{i,j})$ for $\sdp^*(d,k,t)$ with $\sum_{i,j}  \lambda_{i,j} >0$. For the half-integer levels $\sdp(d,k,t+\frac12)$ the dual is formed analogously with $\mathcal M(1)_{2t}$ replaced by $\mathcal M(\{x_{i,j}: i \in [d],j \in [k]\})_{2t+1}$ and ${\Imub}_{,2t}$ replaced by ${\Imub}_{,2t+1}$ (the same point is strictly feasible). 

\subsection{Group-invariance of the MUB problem} \label{sec: groupinvariance}

As mentioned before, the MUB property \eqref{eq: MUB property} of a set of orthonormal bases is preserved under permuting the bases, and under permuting the basis vectors within a single basis. Together these two permutation actions give rise to an action on sets of orthonormal bases of the \emph{wreath product} of $S_d$ and $S_k$. Here we recall the definition of this group and its action on noncommutative polynomials in variables $x_{i,j}$ where $i \in [d]$, $j \in [k]$.

Let $d, k \in \N$. Then the wreath product of $S_d$ and $S_k$ is a group denoted as $S_d \wr S_k$. The group $S_d \wr S_k$ is a finite group with $k!(d!)^k$ elements. Each element is of the form $(\bm{\alpha},\beta) = ((\alpha_1,\ldots,\alpha_k), \beta)$ where $\alpha_i \in S_d$ for each $i \in [k]$ and $\beta \in S_k$. The product of $(\bm{\alpha},\beta)$ and $(\bm{\eta},\zeta)$ is defined as 
\[
(\bm{\alpha},\beta)  (\bm{\eta},\zeta) = (\bm{\alpha}(\beta *\bm{\eta}), \beta \zeta)
\]
where $\beta*\bm{\eta} = (\eta_{\beta^{-1}(1)},\ldots, \eta_{\beta^{-1}(k)})$.\footnote{Equivalently, the wreath product~$S_d \wr S_k$ is the semidirect product~$(S_d)^k \rtimes_{\phi} S_k$ where the homomorphism~$\phi : S_k \to \text{Aut}(S_d)^k)$ comes from the natural action of~$S_k$ on~$(S_d)^k$, i.e., we have~$\phi(\beta)(\bm{\eta}) := \tau*\bm{\eta}  = (\eta_{\tau^{-1}(1)},\ldots, \eta_{\tau^{-1}(k)})$, for~$\beta \in S_k$ and $\bm{\eta} \in (S_d)^k$.}
We define the following action of $S_d \wr S_k$ on noncommutative variables $x_{i,j}$ indexed by pairs $(i,j) \in [d] \times [k]$: for $(\bm{\alpha},\beta) \in S_d \wr S_k$ we set 
\[
(\bm{\alpha},\beta)  x_{i,j} = x_{\alpha_{\beta(j)}(i),\beta(j)}.
\]
This induces an action of $S_d \wr S_k$ on noncommutative polynomials in variables $x_{i,j}$ for $i \in [d]$, $j \in [k]$. 

\paragraph{Group-invariance of the SDP relaxations.}

The above action on monomials makes $\sdp(d,k,t)$ group-invariant for the group~$G=S_d \wr S_k$ and $\sdp(d,k,t+\frac12)$ group-invariant for a suitable subgroup~$G$ of $S_d \wr S_k$ that we describe below. 
Following the discussion in \cref{sec: group invariant pop}, this allows us to restrict to~$G$-invariant~$L$ in the definition of~$\sdp(d,k,t)$ and~$\sdp(d,k,t+\frac12)$. 

For level~$t$, the rows and columns of the moment matrix are indexed by the monomials in~$\R \langle x_{i,j} \, |\, i \in [d], j \in [k]  \rangle_{=t}$. We identify this index set with $Z:=([d]\times [k])^t$. The moment matrix is invariant under the simultaneous action of the wreath product~$G:= S_d \wr S_k$ on its rows and columns. By \cref{prop: symmetry reduction}, to block diagonalize the moment matrix it suffices to find a representative set for the action of~$G$ on~$\C^Z$. We do so in \cref{sec:reprset wreath}.

For level~$t+\frac12$, we have~$dk$ localizing moment matrices (one for each~$x_{i,j}$). By symmetry these moment matrices are all positive semidefinite if and only if the moment matrix expressing~$L(p^*x_{1,1}p) \geq 0$ for each~$p \in \R \langle x_{i,j} \, |\, i \in [d], j \in [k]  \rangle_{=t}$ is positive semidefinite. This moment matrix is indexed by the monomials in~$x_{1,1} \R \langle x_{i,j} \, |\, i \in [d], j \in [k]  \rangle_{=t}$. It can be seen that this set is a module for~$G:=S_{d-1} \times (S_d \wr S_{k-1})$, where~$S_{d-1}$ permutes the indices $2,\ldots,d$ corresponding to~$k=1$ (i.e., the basis elements in the first basis), and where~$S_d \wr S_{k-1}$ acts on the variables~$x_{i,j}$ with $ i \in [d], \, j \in \{2,\ldots, k\}$. 
We show in Section~\ref{sec: representative set thalf} that we can obtain a representative set for this case from~\cref{eq: prodrep} in combination with already derived  representative sets for~$S_{d-1}$-modules (cf.~Sec.~\ref{sec:reprset sk}) and $S_d \wr S_{k-1}$-modules (cf. Sec.~\ref{sec:reprset wreath}). 

\section{Representative set for the action of~\texorpdfstring{$S_k$}{Sk} on \texorpdfstring{$[k]^t$}{[k]t}} \label{sec:reprset sk}

Here we give a representative set for the action of $S_k$ on homogeneous degree-$t$ polynomials in noncommutative variables $x_1,\ldots, x_k$. This serves as a first step towards the decomposition of the $S_d \wr S_k$-invariant SDPs mentioned above in two ways. First, one can obtain $S_k$-invariant relaxations of $\sdp(d,k,t)$ by only considering polynomials in the variables $x_{1,j}$ for $j\in [k]$. The representative set we describe here allows one to symmetry reduce such SDPs. These relaxations are sometimes sufficient to detect non-existence of MUBs, see \cref{sec: numerics}. Second, the construction of a representative set for the $S_d \wr S_k$-action on $([d] \times [k])^t$ builds on the representation theory of the symmetric group. The representative set we construct here applies more broadly to polynomial optimization problems in noncommutative variables $x_1,\ldots,x_k$ that are $S_k$-invariant.

We first provide the necessary preliminaries about the representation theory of the symmetric group. For a more detailed overview we refer the reader to, e.g., Sagan~\cite{Sagan01} (whose notation we follow). In \cref{reprsetSk} we then construct a representative set for the action of $S_k$ on $[k]^t$. In \cref{sec: example 32} we conclude this preliminary section with an example that illustrates the representation theory in the language of polynomials: we show how to apply the theory to the case of homogeneous quadratic polynomials in three variables ($k=3$, $t=2$).

\subsection{Preliminaries on the representation theory of \texorpdfstring{$S_k$}{Sk}} \label{sec: prelim Sk}

\paragraph{Notation.}
Let $k \in \N$.  A \emph{composition $\bm{k}$ of $k$} is a sequence $\bm{k} = (k_1,k_2,\ldots, k_{h})$ of nonnegative integers for which $\sum_{i=1}^{h} k_i = k$. A \emph{partition $\lambda$ of $k$}, denoted $\lambda \vdash k$, is a nonincreasing sequence $\lambda = (\lambda_1,\lambda_2,\ldots, \lambda_{h})$ of positive integers for which $\sum_{i=1}^{h} \lambda_i = k$.  We call $h$ the \emph{height} of $\lambda$, denoted $\mathrm{height}(\lambda)$. The \emph{shape} or \emph{Ferrers diagram} of $\lambda$ is an array containing $k$ cells divided in $\mathrm{height}(\lambda)$ rows such that the $i$th row has $\lambda_i$ cells. As an example, the shape of $(3,2) \vdash 5$ is 
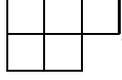
\[
\ytableausetup{centertableaux,boxsize=1.25em}
\begin{ytableau}
\ & \ & \ \\
\ & \
\end{ytableau}.
\] 
A \emph{generalized Young tableau of shape $\lambda$} is an array $\tau$ of shape $\lambda \vdash k$ with the cells filled with positive integers. For brevity, we simply call $\tau$ a tableau. The integers appearing in $\tau$ are called its \emph{content}. The content defines a composition $\mu = (n_1,n_2,\ldots,n_\ell)$ of $k$ where $n_i$ is the number of repetitions of $i$ in $\tau$ and $\ell$ is the largest integer appearing in $\tau$. By concatenating its rows, the tableau $\tau$ defines a $k$-dimensional vector, $\mathrm{vec}(\tau) \in [\ell]^k$, and therefore we may view $\tau$ as a function from $[k]$ to $[\ell]$ (when its shape is clear). We can (equivalently) express the numbers $n_i$ in terms of the function $\tau:[k] \to [\ell]$ as $|\tau^{-1}(i)| = n_i$ for all $i \in [\ell]$. Continuing the above example, we can fill the shape $(3,2) \vdash 5$ with the integers $1,2,3,4,5$ to obtain a (generalized Young) tableau $\tau$ that we can express as an array, a row-vector, or a function 
\[
\tau= \ytableausetup{centertableaux}
\begin{ytableau}
1 & 4 & 2 \\
3 & 5
\end{ytableau}, \qquad \mathrm{vec}(\tau) = (1,4,2,3,5) \in [5]^5, \qquad \begin{tabular}{l|ccccc}
$i$ & $1$ & $2$ & $3$ & $4$ & $5$ \\ 
$\tau(i)$ & $1$ & $4$ & $2$ & $3$ & $5$
\end{tabular}, 
\]
its content defines the composition $(1,1,1,1,1)$, which we abbreviate as $(1^5)$.
Two generalized Young tableaux $\tau,\tau'$ of the same shape are row-equivalent, denoted~$\tau \sim \tau'$, if tableau $\tau'$ can be obtained from $\tau$ by permuting its rows. Given a generalized Young tableau $\tau$, we define the corresponding \emph{tabloid} $\{\tau\}$ as $\{\tau\} = \{\tau' : \tau' \sim \tau\}$.
The tabloid corresponding to the above example is 
\begin{equation} \label{eq:example tabloid}
\{\tau\} = \ \ytableausetup{centertableaux,tabloids}
\begin{ytableau}
1 & 2 & 4 \\
3 & 5
\end{ytableau}. 
\end{equation}
For tabloids we use only horizontal lines to indicate that the order within a row does not matter. By convention we display the elements in a row of a tabloid in non-decreasing order. 
Given a tableau $\tau$ of shape $\lambda \vdash k$ with content $(1^k)$, we let $C_\tau$ be its \emph{column stabilizer}, i.e., the subgroup of $S_k$ that leaves the content of columns invariant. The column stabilizer of our example $\tau$ is $C_\tau = S_{\{1,3\}} \times S_{\{4,5\}} \times S_{\{2\}}$.

A generalized Young tableau is \emph{semistandard} if its rows are (weakly) increasing and its columns are strictly increasing. We let $\widetilde T_{\lambda\mu}$ be the set of tableaux of shape $\lambda$ and content $\mu$ and we let $T_{\lambda\mu}$ be its subset of semistandard tableaux.  
We often use the content 
\[
\mu_r := (k-r,\underbrace{1,\ldots,1}_{r \text{ times}}),
\]
where the integer $k$ is clear from context. Continuing our example, when $\mu = (3,2)$ the set of semistandard Young tableaux is empty when the height of $\lambda$ is at least $3$ (since we only have two distinct elements with which to fill the first column). A complete list of semistandard Young tableaux for $\mu = (3,2)$ thus corresponds to the three partitions $\lambda = (5)$, $(4,1)$, and $(3,2)$:
\begin{align} \label{eq: example semistandard}
\lambda = (5): \ \  \ytableausetup{centertableaux, tabloids=off}  \begin{ytableau}
1 & 1 & 1 & 2 & 2\\
\end{ytableau}, \quad 
\lambda = (4,1): \ \   \ytableausetup{centertableaux}  \begin{ytableau}
1 & 1 & 1 & 2\\
2 \\
\end{ytableau}, \quad 
\lambda = (3,2):  \ \   
\ytableausetup{centertableaux}  \begin{ytableau}
1 & 1 & 1 \\
2 & 2 \\
\end{ytableau}.
\end{align}

\paragraph{Permutation module and Specht modules.}
Two important $S_k$-modules are permutation modules and Specht modules. Let $\lambda \vdash k$ and consider tableaux with content $(1^k) = \{1,2,\ldots,k\}$. The \emph{permutation module} $M^\lambda$ consists of the vector space of tabloids of shape $\lambda$ with content $(1^k)$. The action of $S_k$ on $M^\lambda$ is defined through its action on tableaux. A permutation $\pi \in S_k$ acts on a tableau $\tau$ as $(\pi  \tau)(i) = \pi(\tau(i))$, this induces an action on tabloids: $\pi \{\tau\} = \{\pi  \tau\}$. For example, for the permutation $\pi = (1 \, 3)$ and the tabloid $\tau$ from \cref{eq:example tabloid} we have
\[
\pi \{\tau\} = (1 \, 3) \ \ytableausetup{centertableaux,tabloids}
\begin{ytableau}
1 & 2 & 4 \\
3 & 5
\end{ytableau} = \ytableausetup{centertableaux,tabloids}
\begin{ytableau}
2 & 3 & 4 \\
1 & 5
\end{ytableau}.  
\]
A linear combination of tabloids is called a \emph{polytabloid} and the Specht modules are defined in terms of special polytabloids. Given a tableau $\tau$ of shape $\lambda$ with content $(1^k)$, let 
$\kappa_\tau = \sum_{\pi \in C_\tau} \sign(\pi) \pi$
and define the polytabloid $e_\tau = \kappa_\tau \{\tau\}$, that is, 
\[
e_\tau = \sum_{\pi \in C_\tau} \sign(\pi) \pi \{\tau\}.
\]
The \emph{Specht module} $S^\lambda$ is the submodule of $M^\lambda$ spanned by the polytabloids $e_\tau$ where $\tau \in T_{\lambda, (1^k)}$. The Specht module is a cyclic module and hence it is generated by any of the polytabloids $e_\tau$. From now on we will use as \emph{(canonical) generating element} of $S^\lambda$ the polytabloid $e_t$ corresponding to the tableau $t$ that satisfies $t(i) = i$ for $i \in [k]$.
The Specht modules are the irreducible modules of $S_k$.

\paragraph{Decomposition of a permutation module into Specht modules.} 

We next give the decomposition of a permutation module $M^\mu$ into Specht modules  (see, e.g.,~\cite{Sagan01}). 
Given a semistandard tableau~$\tau\in T_{\lambda \mu} $ and a tabloid~$\{t'\}$ of shape~$\lambda$ with content~$(1^k)$, we define the element\footnote{Note that we deviate here from the notation of \cite{Sagan01} where $\theta_\tau$ is used for the map from $S^\lambda$ to $\C[\widetilde T_{\lambda\mu}]$. The latter is isomorphic to $M^\mu$ and our $\tau \cdot \{t'\}$ corresponds to the composition of $\theta_\tau$ and this isomorphism, applied to $\{t'\}$.}
$$
\tau \cdot \{t'\} := \sum_{\tau'\sim \tau }  \{\tau'* t'\} \in M^{\mu}
$$
of shape~$\mu$ and content~$(1^k)$  
by letting the~$a$-th row  of the tabloid~$\{\tau' * t'\}$  contain all indices $i \in [k]$ for which~$\tau'((t')^{-1}(i)) = a$ (for~$a=1,\ldots,\text{height}(\mu)$).
Note that~$\{\tau' *  t'\}$ is well defined: the tableau~$t'$ has content~$(1^k)$, the tableau~$\tau'$ has content~$\mu$ so there are exactly~$\mu_a$ elements~$i \in [k]$ for which~$\tau'((t')^{-1}(i)) = a$. Moreover, the polytabloid $\tau \cdot \{t'\}$ does not depend on the choice of representative of the equivalence class $\{t'\}$ and hence it is also well defined.\footnote{To see this, suppose that~$t \sim t'$, i.e., $t' = t \circ \pi$, where $\pi \in S_{\lambda_1} \times S_{\lambda_2} \times \cdots \times S_{\lambda_{\mathrm{height}(\lambda)}}$. Then $\{\tau * t'\}$ has on the $a$th row the indices $i \in [k]$ for which $a = \tau((t')^{-1}(i)) = \tau(\pi^{-1} t^{-1} (i))) = \tau' (t^{-1}(i))$, where~$\tau' = \tau \circ \pi^{-1} \sim \tau$, and hence $\{\tau * t'\} = \{\tau' * t\}$. To conclude the argument observe that $\pi$ simply permutes the elements of $\{\tau':\tau' \sim \tau \}$.}

For the (unique) semistandard tableau~$\tau$ of shape $(4,1)$ with content $(3,2)$ and the tableau~$t$ of shape~$(4,1)$ and content~$(1^5)$ corresponding to $t(i)=i$, we have
\[ \{\tau * t \} = \left\{ \,
\ytableausetup{centertableaux,tabloids=off}  \begin{ytableau}
1 & 1 & 1 & 2\\
2 \\
\end{ytableau} \  *  \ 
\ \begin{ytableau}
1 & 2 & 3 & 4 \\
5
\end{ytableau}\ 
\right\}
= \ \ytableausetup{centertableaux,tabloids} \begin{ytableau}
1 & 2 & 3 \\
4 & 5
\end{ytableau},
\]
 and
\begin{align*} \tau \cdot \{t \} =
\ytableausetup{centertableaux,tabloids=off}  \begin{ytableau}
1 & 1 & 1 & 2\\
2 \\
\end{ytableau} \  \cdot  \ \ytableausetup{centertableaux,tabloids}  \begin{ytableau}
1 & 2 & 3 & 4 \\
 5
\end{ytableau}
= \ \ytableausetup{centertableaux,tabloids} \begin{ytableau}
1 & 2 & 3 \\
4 & 5
\end{ytableau} + \ \begin{ytableau}
1 & 2 & 4 \\
3 & 5
\end{ytableau} + \ \begin{ytableau}
1 & 3 & 4 \\
2 & 5
\end{ytableau} + \ \begin{ytableau}
2 & 3 & 4 \\
1 & 5
\end{ytableau}.
\end{align*}

The action of~$\tau$ on tabloids is extended linearly to polytabloids~$e_t$ (which are linear combinations of tabloids~$\{t'\}$) and then to all of~$S^{\lambda}$. In this way,~$\tau \cdot S^{\lambda}$ becomes a submodule of~$M^{\mu}$.
The decomposition of a permutation module $M^\mu$ into Specht modules is now given as (see, e.g.,~\cite{Sagan01}): 
\begin{align} \label{eq: decomp permutation module}
    M^\mu = \bigoplus_{\lambda \vdash k} \Big( \bigoplus_{\tau \in T_{\lambda \mu}} \tau \cdot  S^\lambda \Big).
\end{align}
From \cref{eq: example semistandard} it follows that $M^{(3,2)}$ decomposes as a direct sum of three Specht modules:
\[
M^{(3,2)} = \ \ytableausetup{centertableaux, tabloids=off}  \begin{ytableau}
1 & 1 & 1 & 2 & 2\\
\end{ytableau} \cdot S^{(5)}
\oplus 
\ytableausetup{centertableaux}  \begin{ytableau}
1 & 1 & 1 & 2\\
2 \\
\end{ytableau}  \cdot S^{(4,1)} \oplus \
\ytableausetup{centertableaux}  \begin{ytableau}
1 & 1 & 1 \\
2 & 2 \\
\end{ytableau} \ \cdot S^{(3,2)}.
\] 
The generating element of $\tau \cdot S^{\lambda}$ in $M^\mu$ is obtained by acting with $\tau$ on the generating element $e_t$ of $S^{\lambda}$. That is, we use as generating element of the $\tau$-th copy of $S^\lambda$ the vector 
\begin{equation*} 
v_\tau :=\tau \cdot e_t = \tau \cdot \sum_{c \in C_t} \sign(c) \, \{ct\} =  \sum_{\tau' \sim \tau} \sum_{c \in C_t} \sign(c) \,  \{\tau' * (ct)\}.
\end{equation*}

\subsection{Representative set}\label{reprsetSk}

We now consider the setting of noncommutative polynomials in variables $x_1,\ldots,x_k$ with symmetry coming from $S_k$. Let $t \in \N$ and let $V:= \C^{[k]^t}$. Let $S_k$ act naturally on $[k]^t$ and hence on~$V$. We obtain a first decomposition of $V$ by restricting to $S_k$-orbits. The elements of~$[k]^t / S_k$ correspond bijectively to \emph{set partitions} $P = \{P_1,P_2 \ldots, P_{|P|}\}$ of $[t]$ in at most $k$ parts (i.e., $|P| \leq k$ and $\bigsqcup_{i=1}^{|P|} P_i = [t]$). Each $P_i$ represents a set of indices that are assigned the same $j \in [k]$, and distinct $P_i$ and $P_{i'}$ are assigned distinct $j,j' \in [k]$ respectively. Throughout we let $r  = |P|$ denote the number of sets in the set partition $P$. We then have a first decomposition 
\begin{equation} \label{eq: first decomp of V}
V = \bigoplus_{\substack{P \, \in\, [k]^t / S_k}} V_P
\end{equation}
where $V_P$ is the vector space spanned by $w \in P$. For $r \in [k]$, recall that $\mu_r \vdash k$ is the partition $\mu_r := (k-r,1,\ldots,1)$. 
Then, for a $P \in [k]^t/S_k$ with $P=\{P_1,\ldots,P_r\}$ we have a natural bijection between $w \in P$ and tabloids of shape $\mu_r$ and content $(1^k)$:
\[
w  \quad \xrightarrow{\ \phi_P\ } \quad \ytableausetup{centertableaux,boxsize=1.5em,tabloids} \begin{ytableau}
\mathbf i & \cdots & \mathbf j \\
w(1) \\
w(2) \\
\vdots \\
w(r)
\end{ytableau}
\]
here the first row contains all elements in $[k] \setminus \{w(1),\ldots,w(r)\}$. This bijection respects the action of $S_k$, and therefore we see that $V_P$ is isomorphic to the permutation module $M^{\mu_r}$:
$V_P \cong M^{\mu_r}$.
Combining the decompositions in \cref{eq: first decomp of V,eq: decomp permutation module} gives the decomposition of $V$ into irreducible $S_k$-modules:
\begin{align*}
V  =  \bigoplus_{\substack{P \, \in\, [k]^t / S_k}} \phi_P^{-1} M^{\mu_r}    =  \bigoplus_{\lambda \, \vdash \, k} \left( \bigoplus_{\substack{P \, \in\, [k]^t / S_k}} \bigoplus_{\tau \in T_{\lambda \mu_r}} \phi_P^{-1}(\tau \cdot S^{\lambda}) \right).
\end{align*}
Note that in the above equation we allow $T_{\lambda \mu_r}$ to be empty: it is non-empty only if there exists a semistandard tableau of shape $\lambda$ and content $\mu_r$. Given the shape of $\mu_r$, $T_{\lambda,\mu_r}$ is non-empty if the first row of $\lambda$ has size at least $k-r$ and the height of $\lambda$ is at most $r+1$. This decomposition of $V$ gives rise to a representative set for the action of $S_k$ on $V$. Indeed, for a partition $\lambda \vdash k$, define
\begin{align*}
U_{\lambda}:= \left( u_{\tau,P} \,\,\, | \,\,\, P \in [k]^t/S_k, \,\, \tau \in T_{\lambda, \mu_r}  \right)
\end{align*} 
where
\begin{align} 
u_{\tau,P} := \phi_P^{-1}(v_{\tau}) &= \sum_{\tau' \sim   \tau }\sum_{c \in C_{t}} \mathrm{sgn}(c) \phi_P^{-1}(   \{\tau' *c  t\}). \label{eq: repr element}
\end{align}
Then the set 
$\{U_{\lambda} \,\, | \,\, \lambda \vdash k, \,\,\, \text{height}(\lambda) \leq \min\{k,t+1\} \}$ 
is representative for the action of~$S_k$ on~$V = \C^{[k]^t}$. Indeed, we have shown above that the corresponding irreducible submodules decompose $V$, moreover, by construction the $(\tau,P)$-th copy of $S^\lambda$ is isomorphic to $S^\lambda$ and its generating element is chosen according to this isomorphism. This shows that the elements of $U_\lambda$ satisfy the conditions of \cref{def: representative set}. 

\paragraph{The dimension of \texorpdfstring{$\mathrm{End}_{S_k}(\C^{[k]^t})$}{Sk-invariant endomorphisms on V}.} \label{sec: Sk-invariant endos}

Recall that the sum of the squares of the block sizes equals $\mathrm{dim}(\mathrm{End}_{S_k}(\C^{[k]^t}))$. One can show that the latter equals the number of orbits of pairs $(m_1,m_2)$ where $m_1$ and $m_2$ are noncommutative monomials of degree exactly $t$. Orbits of pairs $(m_1,m_2)$ naturally correspond to orbits of length-$2t$ monomials $m$ via $m = m_1^* m_2$. Hence the dimension of the above space of $S_k$-invariant endomorphisms equals the dimension of the space of $S_k$-invariant noncommutative polynomials of degree exactly~$2t$. As we have seen before (cf.~$V_P$), this equals the number of set partitions of $[2t]$ in at most $k$ parts. When $k \geq 2t$, this number is independent of $k$ and simply becomes the number of set partitions of a set of $2t$ elements. The number of set partitions of a set of $t$ elements is known as the~$t$-th \emph{Bell number}. The sequence of Bell numbers~$\{B_t\}_{t \in \Z_{\geq 0}}$ is known in the online encyclopedia of integer sequences as sequence~\texttt{A000110}; its first ten elements are 	1, 1, 2, 5, 15, 52, 203, 877, 4140, 21147. The sequence can be generated via the recurrence relation
$B_t = \sum_{j=0}^{t-1} \binom{t-1}{j} B_j$. We emphasize that for $k \geq 2t$ the sum of the squares of the block sizes of the symmetry reduced SDP equals $B_{2t}$ and is thus independent of $k$.

\paragraph{Computing the coefficients of the block-diagonalized SDP.}
To compute the coefficients of the block-diagonalized semidefinite program we may use the following formula.  Fix a partition $\lambda \vdash k$ and $u_{\tau,P}, u_{\sigma,Q} \in U_{\lambda}$, then for any $A \in \mathrm{End}_{S_k}(V)$ we have,
\begin{equation} \label{eq: inner prod Sk}
\begin{split}
    &\langle u_{\tau,P}, A u_{\sigma,Q} \rangle = \sum_{\substack{\tau' \sim \tau, \\ \sigma'\sim \sigma}} \sum_{c,c' \in C_{t}} \mathrm{sgn}(cc') A_{ \phi_P^{-1}( \tau' *  \{c  t\}), \phi_Q^{-1}( \sigma' *  \{c'  t\})}
                \\&  = \sum_{\substack{\tau' \sim \tau, \\ \sigma' \sim \sigma}} \sum_{c,c' \in C_{t}} \mathrm{sgn}(cc') A_{\phi_P^{-1}( \tau' *  \{ t\}), \phi_Q^{-1}( \sigma' *  \{c^{-1}c'  t\})}
    = |C_t|\sum_{\substack{\tau' \sim \tau, \\ \sigma'\sim \sigma}} \sum_{c \in C_{t}} \mathrm{sgn}(c) A_{\phi_P^{-1}( \tau' * \{ t\}),\phi_Q^{-1}( \sigma' *  \{c  t\})}
\end{split}
\end{equation}   
where in the last equality we use the group structure of $C_t$. 

At first glance, even for fixed level $t$, \cref{eq: inner prod Sk} could take time exponential in $k$ to evaluate since the number of terms can be exponentially large (take for instance $\lambda =(t^{k/t})$). However, one can show that it is possible to compute these inner products in time polynomial in~$k$ (for fixed level~$t$). 
Define commutative variables $z_{i,j}$ for $i \in [r_P+1], j \in [r_Q+1]$. Let $f_{\tau,\sigma}$ be the polynomial 
\begin{equation*}
    f_{\tau,\sigma}( z) = \sum_{\substack{\tau' \sim \tau, \\ \sigma' \sim \sigma}} \sum_{c,c' \in C_\lambda} \mathrm{sgn}(cc') \prod_{y \in [k]} z_{c\cdot \tau'(y), c'\cdot\sigma'(y)}. 
\end{equation*}
Consider a monomial $m = \prod_{y \in [k]} z_{i_y,j_y}$. A variable $z_{i,j}$ for $i,j>1$ appears in $m$ if $P_{i-1}$ and $Q_{j-1}$ are assigned the same element in $[k]$, and hence $m$ represents an $S_k$-orbit of $[k]^{t} \times [k]^t$. The coefficient of $m$ in $f_{\tau,\sigma}$ thus counts the number of $c,c' \in C_t$ and $\tau' \sim \tau, \sigma' \sim \sigma$ for which $(\phi_P^{-1}( \tau' *  \{c  t\}),\phi_Q^{-1}( \sigma' *  \{c'  t\}))$ are $S_k$-equivalent.
It was shown in~\cite[Sec.~3]{Gij09} that~$f_{\tau,\sigma}$ can be expressed as a linear combination of monomials in time polynomial in $k$, assuming the height of~$\tau,\sigma$ is fixed (which holds in our case since~$\max\{r_P+1, r_Q+1\}\leq t+1$). See also \cite[Prop.~3]{LPS17} for a different proof.

\subsection{
An example: homogeneous quadratic polynomials}
\label{sec: example 32}

To illustrate the concepts developed in the previous sections, we show how to use representation theory to block-diagonalize moment matrices indexed by homogenous quadratic polynomials in three variables that are invariant under the action of $S_3$. We express the concepts in the language of polynomials. (This example corresponds to the case $k=3$ and $t=2$.) 

For any~$L \in \R \langle x_1,x_2,x_3\rangle_{=4}^*$, we write~$M(L)$ for the moment matrix whose rows and columns are indexed by elements of~$\langle x_1,x_2,x_3\rangle_{=2}$. We show how to block-diagonalize the algebra of such matrices $M(L)$ that are invariant under the simultaneous action of~$S_3$ on their rows and columns. We can identify this algebra with $(\C^{[3]^2 \times [3]^2})^{S_3}=\mathrm{End}_{S_3}(\C^{[3]^2})$,  the algebra of $[3]^2 \times [3]^2$-matrices that are invariant under the same action. As discussed in \cref{sec:symmetry block}, block-diagonalizing this algebra amounts to decomposing the $S_3$-module $V :=\C^{[3]^2}$ into irreducible modules. This module can be identified with the space of trivariate homogeneous quadratics 
\[
 \C^{[3]^2} \simeq \C\langle x_1,x_2,x_3\rangle_{=2}
\]
where the group $S_3$ acts on this space by permuting the indices of $x_1,x_2,x_3$. 

We now follow the recipe provided in \cref{reprsetSk} to obtain this decomposition. The first step is to decompose $V$ into orbits of $S_3$. A permutation $\sigma \in S_3$ maps a monomial $x_i x_j$ to $x_{\sigma(i)} x_{\sigma(j)}$. It is then not hard to see that we have only two orbits: one formed by monomials with $i=j$ and one where $i \neq j$. This corresponds to \cref{eq: first decomp of V} from the previous section: we write 
\[
V =\C^{[3]^2}\cong \C \langle x_1,x_2,x_3 \rangle_2 = V_{11} \oplus V_{12}, 
\]
where~$V_{11} = \Span\{x_i^2\,: \, i \in [3]\}$   and $V_{12} = \Span\{x_i x_j \,:\, i,j \in [3] \text{ with } i \neq j\}$. 

In the language of the previous section, the space~$V_{11}$ corresponds to the partition~$P=\{\{1,2\}\}$ of~$[t]=[2]$, or the~$S_3$ orbit of~$(1,1)\in [k]^t$. We have~$V_{11} \cong M^{\mu_{1}}$ where the natural bijection between elements in this orbit and tabloids of shape~$\mu_1=(2,1)\vdash 3$ and content~$(1^3)$ is as follows: 
\begin{align*}
   x_1^2 \leftrightarrow \ytableausetup{centertableaux,boxsize=1.5em,tabloids} \begin{ytableau}
2 & 3 \\
1 \\
\end{ytableau}, \,\,\,\,    x_2^2 \leftrightarrow \ytableausetup{centertableaux,boxsize=1.5em,tabloids} \begin{ytableau}
1 & 3 \\
2 \\
\end{ytableau}, \,\,\,\,     x_3^2 \leftrightarrow \ytableausetup{centertableaux,boxsize=1.5em,tabloids} \begin{ytableau}
1 & 2 \\
3 \\
\end{ytableau}.
\end{align*}
Similarly, the space~$V_{12}$ corresponds to the partition~$P=\{\{1\},\{2\}\}$ of~$[t]=[2]$, or the~$S_3$ orbit of~$(1,2)\in [k]^t$. We have~$V_{12} \cong M^{\mu_{2}}$ where the bijection between elements in this orbit and tabloids of shape~$\mu_2=(1,1,1)\vdash 3$ and content~$(1^3)$ is given by: 
\begin{align*}
   x_ix_j\,\, \leftrightarrow \,\,\ytableausetup{centertableaux,boxsize=1.5em,tabloids}
   \begin{ytableau}
k \\ i \\
j \\
\end{ytableau}, \quad \text{for $i\neq j \in [3]$ and~$k$ the remaining element~$\neq i,j$ in~$[3]$}.
\end{align*} 
The advantage of identifying the orbits with these permutation modules is that we can use the representation theory of the symmetric group to further decompose $V_{11}$ and $V_{12}$ into irreducible modules. Here we translate the theory to the language of polynomials. The irreducible modules of $S_k$ are indexed by all partitions $\lambda$ of $k$, which for the case $k=3$ amounts to the partitions $(3)$, $(2,1)$ and $(1,1,1)$. The blocks of the block-diagonalized algebra will be labeled by these partitions and we can thus already note that there will be three blocks.  

As a concrete example, we compute one of the representative elements that correspond to the decomposition of $V_{11}$. 
Following \cref{eq: decomp permutation module}, the space~$V_{11}$ decomposes as 
$$
\bigoplus_{\lambda \vdash 3} \bigoplus_{\tau \in T_{\lambda\mu_1}} \phi^{-1}_{11}(\tau \cdot S^{\lambda}).
$$
This decomposition is labeled by partitions $\lambda \vdash 3$ and semistandard tableaux $\tau \in T_{\lambda \mu_1}$. In this small example this reduces to a decomposition into two parts: there are unique semistandard tableaux with content $\mu_1 = (2,1)$ for the partitions $(3)$ and $(2,1)$: 
\begin{align*} 
\lambda=(3):\ \ytableausetup{centertableaux,boxsize=1.5em,tabloids=off} \begin{ytableau}
1 & 1  & 2\\
\end{ytableau}\ ; \,\,\,\,    
\lambda=(2,1): \ \ytableausetup{centertableaux,boxsize=1.5em}  \begin{ytableau}
1 & 1\\
2 \\
\end{ytableau}\ .
\end{align*} 
There is no semistandard tableau of shape $(1,1,1)$ with content $(2,1)$. 
In the language of the previous section, $V_{11}$ corresponds to the partition $P = \{\{1,2\}\}$. We show how to compute the representative element $u_{\tau,P}$ for 
this $P$ and the semistandard tableau corresponding to the partition~$(2,1)$. 
We have 
\begin{align*}
u_{\tau,P} &= \sum_{\tau' \sim   \tau }\sum_{c \in C_{t}} \mathrm{sgn}(c) \phi_P^{-1}( \tau' *  \{c  t\})
\\&=  \phi_{11}^{-1}( \tau *  \{ t\}) +\text{sgn}( (13)) \phi_{11}^{-1}( \tau *  \{(1 3)  t\})
\\&=  \phi_{11}^{-1}\left(  \ytableausetup{tabloids,boxsize=1.5em}  \ \begin{ytableau}
1 & 2\\
3 \\
\end{ytableau}\ \right) - \phi_{11}^{-1}\left(  \ytableausetup{tabloids,boxsize=1.5em}  \ \begin{ytableau}
3 & 2\\
1 \\
\end{ytableau}\ \right) = x_3x_3-x_1x_1,
\end{align*} 
where for the second equality we use that there is a unique tableau that is row-equivalent to $\tau$ (namely $\tau$ itself) and the column-stabilizer corresponding to this shape (and its standard labeling~$t$) consists of the identity permutation and the permutation $(1 \, 3)$. For completeness we mention that the other representative element corresponding to $V_{11}$, the one associated to $\tau=   \ytableausetup{centertableaux,tabloids=off,boxsize=1.5em}  \begin{ytableau}
1 & 1 &2\\
\end{ytableau}$, is $x_1x_1+x_2x_2+x_3x_3$. 

For the space~$V_{12}$ one can do the same, the space corresponds to the partition~$P=\{\{1\},\{2\}\}$ of~$[t]=[2]$, or the~$S_3$ orbit of~$(1,2)\in [k]^t$. One computes representative elements~$u_{\tau,P}$ for each of the $\phi^{-1}_{P}(\tau \cdot S^{\lambda})$ with $P=\{\{1\},\{2\}\}$ via the procedure illustrated above where we now have the following semistandard tableaux for the three choices of partitions $\lambda$: 
\[
\lambda = (3): \ \ytableausetup{centertableaux,boxsize=1.5em,tabloids=off} \begin{ytableau}
1 & 2  & 3\\
\end{ytableau}\ ; \,\,\,\,    
\lambda = (2,1): \ \ytableausetup{centertableaux,boxsize=1.5em}  \begin{ytableau}
1 & 2\\
3 \\
\end{ytableau}\ , \,\,\,\,    \ytableausetup{centertableaux,boxsize=1.5em}  \begin{ytableau}
1 & 3\\
2 \\
\end{ytableau}\ ; \,\,\,\,    \lambda = (1,1,1): \ \ytableausetup{centertableaux,boxsize=1.5em}  \begin{ytableau}
1 \\
2 \\
3\\
\end{ytableau} \ .
\]

A representative set for the action of~$S_3$ on~$\C^{[3]^2}$ is then denoted by~$\{U_{\lambda} \, | \, \lambda \vdash 3, \, \text{height}(\lambda) \leq 3 \} $, and it is obtained by grouping the representative elements for~$V_{11}$ and~$V_{12}$ according to their partition $\lambda$. For example, both $V_{11}$ and $V_{12}$ contain a single irreducible sub-module that is isomorphic to $S^{(3)}$ and thus $U_{(3)}$ contains two elements; these correspond to the polynomials 
\[
\sum_{i \in [3]} x_ix_i \text{ and  } \sum_{\substack{(i,j) \in [3]\\ i \neq j}} x_ix_j.
\]Formally, we can view the representative set as a collection of matrices by identifying a polynomial with its (column) vector of coefficients; $U_{(3)}$ is a $9 \times 2$ matrix. 
The remaining two matrices, $U_{(2,1)}$ and $U_{(1,1,1)}$ have columns corresponding to the following polynomials: 
\begin{align*} 
U_{(2,1)} :& \qquad x_3x_3-x_1x_1, \quad x_2x_3-x_2x_1+x_1x_3-x_3x_1, \quad x_3x_2-x_1x_2+x_3x_1-x_1x_3,\\
U_{(1,1,1)} :& \qquad x_1x_2 + x_2x_3+ x_3x_1-x_2x_1-x_1x_3-x_3x_2.
\end{align*}

Having computed the representative set, we finally obtain a block-diagonalization of~$(\C^{[3]^2 \times [3]^2})^{S_3}$ by computing (cf.~\cref{prop: symmetry reduction,eq: repr element})
\[
U_{(3)}\T M(L) U_{(3)} \oplus U_{(2,1)}\T M(L) U_{(2,1)}  \oplus U_{(1,1,1)}\T M(L) U_{(1,1,1)}.
\]
The matrices~$U$ technically contain coefficient vectors of polynomials and the moment matrix~$M(L)$ contains the values of~$L$ on monomials, but by viewing the columns of each representative matrix~$U_\lambda$ as a polynomial, we can compute the~$\ell j$-th entry~$U_{\lambda,\ell} \T M(L) U_{\lambda,j}$ of the block corresponding to $\lambda$ by expanding~$L(U_{\lambda,\ell}^*U_{\lambda,j})$. Here $U_{\lambda,j}$ and $U_{\lambda,\ell}$ are the polynomials corresponding to columns $j$ and $\ell$ of~$U_\lambda$, respectively.
As an example, we display the $(1,2)$-entry of the block corresponding to~$U_{(3)}$:
\begin{align*}
U_{(3),1}\T M(L) U_{(3),2} = L\left(U_{(3),1}^* U_{(3),2}\right) &=L\Big(\big(\sum_{i \in [3]} x_ix_i\big)^*\big(\sum_{\substack{(i,j) \in [3]\\ i \neq j}} x_ix_j\big)\Big)\\&=  6(L(x_1x_1x_1x_2)+L(x_1x_1x_2x_1) +L(x_1x_1x_2x_3)),
\end{align*}
where for the last equality we use that $L$ is constant on orbits. This allows us to express the entry as a linear combination of the linear functional evaluated at orbit representatives (here we use as orbit representative the lexicographically smallest element in the $S_3$-orbit).\footnote{For completeness we mention that the case~$t=2$ and general~$k$ was studied in~\cite[Sec.~3.1, App.~2]{LPS17}, where a representative set for the action of~$S_k$ on~$[k]^2$ together with the explicit formulas in the corresponding block-diagonalization are given. }

\section{Preliminaries on the representation theory of \texorpdfstring{$S_d \wr S_k$}{Sd wr Sk}} \label{sec: rep wreath}

The irreducible modules of $S_d \wr S_k$ are well known, see for example~\cite{Ker71, MacDonald80}. Here we collect the necessary definitions and constructions that lead to the irreducible modules for the wreath product of $S_d \wr S_k$. 

\paragraph{Notation.}
In the representation theory of $S_k$ partitions and compositions play a large role, for the wreath product we moreover need the notion of \emph{multipartitions}. Let $\ell \in \N$, then an $\ell$-multipartition of $k$ is an $\ell$-tuple $\underline \Lambda = (\Lambda^1,\ldots, \Lambda^\ell)$ such that each $\Lambda^a$ is a partition and $\sum_{a \in [\ell]} |\Lambda^a| = k$. We write $\underline \Lambda \vdash k$ and, throughout, underline each symbol that corresponds to an $\ell$-multipartition. To an $\ell$-multipartition $\underline \Lambda$ we associate the \emph{$\ell$-composition} $|\underline \Lambda| = (|\Lambda^1|,\ldots, |\Lambda^\ell|)$ of $k$. 

\paragraph{Induction.}
Two natural modules of $S_d \wr S_k$ involve the notion of induction which we recall here for general groups for convenience and to fix notation. Let~$G$ be a group and~$H$ be a subgroup of~$G$. If~$V$ is a module of~$H$, we define the \emph{induced module} $\left. V \vphantom{\sum}\right\uparrow_H^G$ as follows.  Let~$R:=\{r_1,\ldots,r_s\}$ be a set consisting of representatives of the (left) cosets of~$H$ in~$G$, so $|R|=[G  : H]$. Then the elements of~$\left. V \vphantom{\sum}\right\uparrow_H^G$ are sums of the form $\alpha_1 (r_1,v_1) + \ldots + \alpha_s (r_s,v_s)$, for $v_1,\ldots,v_s \in V$ and scalars~$\alpha_1,\ldots,\alpha_s \in \C$. 
 An element $g\in G$ acts on an element~$(r_i, v)$ via $g  (r_i, v) = (r_j, h v)$, where~$r_j \in R$ and~$h \in H$ are the unique elements such that~$gr_i = r_jh$. This action is extended linearly to $\left. V \vphantom{\sum}\right\uparrow_H^G$. 

\paragraph{Basic constructions of modules.}
Modules of $S_d$ and $S_k$ can be used to construct modules for $S_d\wr S_k$. We list some essential constructions. Let $X$ be an $S_d$-module. We define an $S_d \wr S_k$-module $X^{\tilde \boxtimes k}$ as follows: as vector space it equals $X^{\otimes k}$ and the action of an element $(\bm \alpha,\beta) = (\alpha_1,\ldots,\alpha_k;\beta) \in S_d \wr S_k$ on a rank-1 tensor $\otimes_{i \in [k]} x_i$ is 
\[
(\alpha_1,\ldots,\alpha_k;\beta)  \bigotimes_{i \in [k]} x_i = \bigotimes_{i \in [k]} \alpha_{i}  x_{\beta^{-1}(i)}.
\]
A similar construction can be used to create a module for the wreath product between $S_d$ and a \emph{Young subgroup} of $S_k$. A Young subgroup of $S_k$ associated to an $\ell$-composition $\bm k$ of $k$ is the subgroup $S_{\bm k} := S_{\{1,\ldots,k_1\}} \times \cdots \times S_{\{1+\sum_{a < \ell} k_a, k\}}$ of $S_k$. 
For $S_d$-modules $X_1,\ldots, X_\ell$ we define the $S_d \wr S_{\bm k}$-module $(X_1,\ldots,X_\ell)^{\tilde \boxtimes \bm k}$ as follows:
\[
(X_1,\ldots,X_\ell)^{\tilde \boxtimes \bm k} := X_1^{\tilde \boxtimes k_1} \otimes X_2^{\tilde \boxtimes k_2} \otimes \cdots \otimes X_\ell^{\tilde \boxtimes k_\ell},
\]
where $(\bm \alpha, \beta) = (\alpha_1,\ldots, \alpha_k,\beta) \in S_d \wr S_{\bm k}$ acts on an element $\otimes_{i \in [k]} x_i$ as 
\begin{equation} \label{eq: wreath action on boxtimes}
(\bm \alpha; \beta)  \bigotimes_{i \in [k]} x_i = \bigotimes_{i \in [k]} \alpha_i  x_{\beta^{-1}(i)}. 
\end{equation}
In what follows, it will be convenient to work with a module that is closely related to the above $S_d \wr S_{\bm k}$-module, but where the tensor legs are permuted. For a tuple $\bm j \in [\ell]^k$ we consider the Young subgroup 
$S_{\bm j} := S_{\bm j^{-1}(1)} \times \cdots \times S_{\bm j^{-1}(\ell)}$
and we define the $S_d \wr S_{\bm j}$-module 
\begin{equation*} 
(X_1,\ldots,X_\ell)^{\tilde \boxtimes \bm k(\bm j)}_{\bm j}
\end{equation*}
which equals $\bigotimes_{i \in [k]} X_{j(i)}$ as vector space. The $S_d \wr S_{\bm j}$ action is defined by identifying, for each $a \in [\ell]$, the module $(X_1,\ldots,X_\ell)^{\tilde \boxtimes \bm k}_{\bm j}$ restricted to the coordinates for which $j(i)=a$ with the $S_d \wr S_{\bm j^{-1}(a)}$-module~$X_a^{\tilde \boxtimes k_a}$. That is, $(\bm \alpha, \beta) \in S_d \wr S_{\bm j}$ acts on an element $\otimes_{i \in [k]} x_i$ exactly as in \eqref{eq: wreath action on boxtimes}.
The following construction is used to combine $S_d \wr S_k$-modules and $S_k$-modules.

\begin{definition}
Let~$H$ be a subgroup of~$S_k$, $G$ a subgroup of~$S_d$, $Y$ an~$H$-module, and~$Z$ a $(G \wr H)$-module.
Then define $Z \oslash Y := Z \otimes Y$ as vector space, with action
$(\bm \alpha, \beta)  (x,y) = ((\bm \alpha,\beta)  x ,\beta y)$,
so that~$Z \oslash Y$ is a~$(G \wr H)$-module.\footnote{Alternatively, it holds that~$Z \oslash Y= Z \otimes \mathrm{Inf}_H^{G\wr H} Y$, where the \emph{inflation} $\mathrm{Inf}_H^{G\wr H} Y$ of~$Y$ from $H$ to~$G \wr H$ is~$Y$ as vector space, with action $(\bm \alpha, \beta) y = \beta y$.}
\end{definition}

\paragraph{Analogues of Specht modules for $S_d \wr S_k$.}
In what follows, let $\nu_1,\ldots, \nu_{\ell} \vdash d$ be a complete list of partitions of $d$, sorted lexicographically. So $\ell$ is the number of partitions of $d$. For induction, we use $d \wr k$ as shorthand for $S_d \wr S_k$ (and similarly for $d \wr \bm j$ and $d \wr |\Lambda|$). Using the above constructions, one can define an analogue of Specht modules for the wreath product $S_d \wr S_k$. 
\begin{definition} \label{def: Specht module wreath}
For an $\ell$-multipartition $\underline \Lambda$ of $k$, the \emph{Specht module} $S^{\underline \Lambda}$ is defined as 
\[
S^{\underline \Lambda} := \Big[(S^{\nu_1}, \ldots, S^{\nu_{\ell}})^{\tilde \boxtimes |\underline\Lambda|} \oslash (S^{\Lambda^1} \boxtimes \cdots \boxtimes S^{\Lambda^\ell})\Big]\left.\vphantom{\Big]}\right\uparrow_{d \wr |\underline \Lambda|}^{d \wr k}, 
\]
and for $\bm j \in [\ell]^k$ with $\bm j^{-1}(a) = |\Lambda^a|$ (for all $a \in [\ell]$) we define the module $S_{\bm j}^{\underline \Lambda}$ as 
\[
S^{\underline \Lambda}_{\bm j} := \Big[(S^{\nu_1}, \ldots, S^{\nu_{\ell}})^{\tilde \boxtimes |\underline\Lambda|}_{\bm j} \oslash (S^{\Lambda^1} \boxtimes \cdots \boxtimes S^{\Lambda^\ell})\Big]\left.\vphantom{\Big]}\right\uparrow_{d \wr \bm j}^{d \wr k}.
\]
\end{definition}
The modules $S^{\underline \Lambda}$ form a complete set of irreducible $S_d\wr S_k$-modules (cf.~\cite{MacDonald80,CT03}) and therefore can be used to decompose any (finite-dimensional) $S_d \wr S_k$-module. Observe that $S^{\underline \Lambda} = S^{\underline \Lambda}_{\bm j}$ for the $k$-tuple $\bm j$ for which, for each $a \in [\ell]$ with $|\Lambda^a|>0$, we have $\bm j^{-1}(a) = \{1+\sum_{j<a} |\Lambda^a|, \sum_{j \leq a} |\Lambda^a|\}$. Below in \cref{lem: with equals without j} we show that in fact for every $\bm j$ with $|\bm j^{-1}(a)| = |\Lambda^a|$ we have that $S^{\underline \Lambda}$ and $S^{\underline \Lambda}_{\bm j}$ are $S_d \wr S_k$-isomorphic. We first describe the generating elements of $S^{\underline \Lambda}_{\bm j}$. 

For $\bm j \in [\ell]^k$ the canonical generating element of $S^{\underline \Lambda}_{\bm j}$ is obtained by combining the generating elements of the respective Specht modules of $S_d$ and $S_k$ as follows. We define for $i \in [k]$ and $a \in [\ell]$
\[
v_i = \kappa_{t_{{\bm j}(i)}}\{t_{{\bm j}(i)}\}, \qquad u_a = \kappa_{t_{{\Lambda^a}}} \{t_{\Lambda^a}\}
\]
 where $t_{{\bm j}(i)}$ is the usual fixed tableau of shape $\nu_{{\bm j}(i)}$ and where $t_{\Lambda^a}: \bm j^{-1}(a) \to \bm j^{-1}(a)$ is the standard tableau of shape $\Lambda^a$. Throughout, as coset representatives of $S_d \wr S_{\bm j}$ in $S_d\wr S_k$ we consider the elements $(\id,b) \in S_d \wr S_k$ for which $b \in S_k$ is an increasing function on each $\bm j^{-1}(a)$, i.e., $b$ is such that if $i,i' \in \bm j^{-1}(a)$ with $i<i'$ then $b(i) < b(i')$. The canonical generating element of $S^{\underline \Lambda}_{\bm j}$ is the vector $\big((\id,\id),(\left(\bigotimes_{i \in [k]} v_i\right) \otimes \left(\bigotimes_{a \in [\ell]} u_a\right)\big)$. To emphasize the difference between $S_d$- and $S_k$-modules we write 
\[
\big((\id,\id),(v_1,\ldots,v_k ; u_1,\ldots,u_\ell )\big) := \Big((\id,\id),\Big(\Big(\bigotimes_{i \in [k]} v_i\Big) \otimes \Big(\bigotimes_{a \in [\ell]} u_a\Big)\Big)\Big).
\]

\begin{lemma} \label{lem: with equals without j}
Let $\underline \Lambda$ be an $\ell$-multipartition of $k$ and let $\bm j \in [\ell]^k$ be such that $|\bm j^{-1}(a)|=|\Lambda^a|$ for all $a \in [\ell]$. Then the $S_d \wr S_k$-module $S^{\underline \Lambda}$ is isomorphic to the $S_d \wr S_k$-module $S^{\underline \Lambda}_{\bm j}$. 
\end{lemma}
\begin{proof}
For $a \in [\ell]$ with $|\Lambda^a|>0$, let $I_a:= \{1+\sum_{j<a} |\Lambda^a|, \sum_{j \leq a} |\Lambda^a|\}$. Let $\pi \in S_k$ be the permutation that maps the $t$-th smallest element in $I_a$ to the $t$-th smallest element in $\bm j^{-1}(a)$ for each $a \in [\ell]$ and $t \in [|\Lambda^a|]$. We show this permutation defines an isomorphism between $S^{\underline 
\Lambda}$ and $S^{\underline \Lambda}_{\bm j}$: 
\[
\big((\id,b),(v_1,\ldots, v_k; u_1,\ldots, u_\ell)\big) \xmapsto{\phi} \big((\id,b\pi^{-1}),(v_{\pi^{-1}(1)},\ldots, v_{\pi^{-1}(k)}; \pi u_1,\ldots, \pi u_\ell)\big) 
\]
Observe that $(\id,b\pi^{-1})$ is indeed a coset representative of $S_d \wr S_{\bm j}$ in $S_d \wr S_k$: $b\pi^{-1} \in S_k$ and for $i, i' \in \bm j^{-1}(a)$ with $i<i'$ we have $\pi^{-1}(i) <  \pi^{-1}(i')$ and therefore $b\pi^{-1}(i) <  b\pi^{-1}(i')$.  Moreover, the set of coset representatives $(\id,b)$ of $S_d \wr S_{|\underline \Lambda|}$ is mapped bijectively to the set of coset representatives of $S_d \wr S_{\bm j}$ in $S_d \wr S_k$. Finally, for each coset the map is bijective (since it acts as a permutation).

We now show that the $S_d \wr S_k$ action is preserved. Let $(\bm \alpha, \beta) \in S_d \wr S_k$. By definition 
\begin{align*}
(\bm \alpha, \beta) \big((\id,b),(v_1,\ldots, v_k; u_1,\ldots, u_\ell)\big) &= \big((\id,\tilde b),(\tilde \alpha_1 v_{\tilde \beta^{-1}(1)},\ldots, \tilde \alpha_k v_{\tilde \beta^{-1}(k)}; \tilde \beta u_1,\ldots, \tilde \beta u_\ell)\big)
\end{align*}
where $(\id, \tilde b)$ is the unique coset representative such that $(\bm \alpha, \beta) (\id,b) = (\id, \tilde b) (\bm{\tilde \alpha},\tilde \beta)$ with $(\bm{\tilde \alpha},\tilde \beta) \in S_d \wr S_{|\Lambda|}$. Hence 
\begin{align*}
& \phi\big( (\bm \alpha, \beta) \big( (\id,b),(v_1,\ldots, v_k; u_1,\ldots, u_\ell)\big)\big) \\
&= \big((\id,\tilde b \pi^{-1}),(\tilde \alpha_{\pi^{-1}(1)} v_{\tilde \beta^{-1}\pi^{-1}(1)},\ldots, \tilde \alpha_{\pi^{-1}(k)} v_{\tilde \beta^{-1}\pi^{-1}(k)}; \pi \tilde \beta u_1,\ldots, \pi \tilde \beta u_\ell)\big).
\end{align*}
On the other hand, using that $(\bm \alpha,\beta)(\id,b) = (\id,\tilde b)(\bm{\tilde \alpha},\tilde \beta)$ one can verify that $(\bm \alpha,\beta)(\id,b \pi^{-1}) = (\id,\tilde b \pi^{-1}) (\pi \bm{\tilde \alpha},\pi \tilde \beta\pi^{-1})$. Then we have 
\begin{align*}
 (\bm \alpha, \beta)    \phi &\big((\id,b),(v_1,\ldots, v_k; u_1,\ldots, u_\ell)\big) \\
 &= (\bm \alpha, \beta)\big((\id,b\pi^{-1}),(v_{\pi^{-1}(1)},\ldots, v_{\pi^{-1}(k)};  \pi u_1,\ldots,   \pi u_\ell)\big)   \\&= \left( (\id,\tilde b \pi^{-1}),(\pi \tilde{\alpha}, \pi \tilde{\beta} \pi^{-1})(v_{\pi^{-1}(1)},\ldots, v_{\pi^{-1}(k)};  \pi u_1,\ldots,   \pi u_\ell) \right)
    \\&= \big((\id,\tilde b \pi^{-1}),(\tilde \alpha_{\pi^{-1}(1)} v_{\tilde \beta^{-1}\pi^{-1}(1)},\ldots, \tilde \alpha_{\pi^{-1}(k)} v_{\tilde \beta^{-1}\pi^{-1}(k)}; \pi \tilde \beta u_1,\ldots, \pi \tilde \beta u_\ell)\big).
\end{align*}
So $ \phi$ respects the~$S_d \wr S_k$-action.
\end{proof}

\section{Representative set for the action of~\texorpdfstring{$S_d \wr S_k$}{Sd wreath Sk} on \texorpdfstring{$([d] \times [k])^t$}{([d]-times-[k])t}} \label{sec:reprset wreath}

We now turn our attention to the setting of noncommutative polynomials in variables $x_{i,j}$ with $i \in [d], j \in [k]$. We consider the natural action of the wreath product $S_d \wr S_k$ on such polynomials. We construct a representative set for the action of $S_d \wr S_k$ on $\C^{([d]\times[k])^t}$. 

As before, we start by decomposing $\C^{([d] \times [k])^t}$ using the equivalence classes $([d]\times[k])^t / (S_d \wr S_k)$. Such an equivalence class corresponds to a pair $(P,\mathbf Q)$ where $P = \{P_1,\ldots, P_r\}$ is a set partition of $[t]$ in at most $k$ parts and $\mathbf Q = (Q_1,\ldots,Q_r)$ where each $Q_i$ is a set partition of $P_i$ in at most $d$ parts. This gives a first decomposition of our vector space:  
\begin{equation} \label{eq: first decomposition of V}
\C^{([d] \times [k])^t} = \bigoplus_{\substack{(P,\mathbf Q) \in ([d]\times[k])^t / (S_d \wr S_k)}} V_{P,\mathbf Q} 
\end{equation}
where $V_{P,\mathbf Q}$ is the vector space spanned by elements in the equivalence class $(P,\mathbf Q)$. For~$P$ and $\bm{Q} =(Q_1,\ldots,Q_r)$, set~$q_i:=|Q_i|$ for each~$i \in [r]$. The vector space $V_{P,\bm Q}$ is spanned by monomials $x_{i_1,j_1} \cdots x_{i_t,j_t}$ for which the set partition $P$ is used to assign \emph{distinct} bases ($j_a$'s) to variables in the monomial and the set partition $Q_i$ of $P_i$ is then used to assign $q_i = |Q_i|$ \emph{distinct} basis elements ($i_a$'s) to the variables in $P_i$. Below in \cref{sec: decomp of VPQ} we use the representation theory of the symmetric group to obtain a decomposition of each of the $S_d \wr S_k$-modules $V_{P,\bm Q}$. In \cref{sec: iso to Specht modules}, we show that this is in fact a decomposition into \emph{irreducible} $S_d \wr S_k$-modules by constructing isomorphisms to Specht modules~$S^{\underline \Lambda}$. In \cref{sec: explicit decomposition} we give an explicit description of the representative set. We conclude the section with a worked out example of one representative element in the language of polynomials, see \cref{sec: example 232}.

\subsection{\texorpdfstring{$V_{P,\bm Q}$}{VPQ} is a module for \texorpdfstring{$S_d \wr S_k$}{Sd wr Sk}} \label{sec: decomp of VPQ}
We give a natural interpretation of $V_{P,\bm Q}$ in terms of tabloids. In what follows we use the partitions $\mu_r^{(k)} = (k-r,1^r)$ and $\mu_q^{(d)} = (d-q, 1^q)$. 
Consider the vector space 
\begin{equation*}
    W := (M^{\mu^{(d)}_{q_1}} \otimes \ldots \otimes M^{\mu^{(d)}_{q_r}}) \otimes  M^{\mu^{(k)}_r}. 
\end{equation*}
The standard basis vectors of $W$ are naturally associated to tuples of standard basis vectors $v_i \in M^{\mu^{(d)}_{q_i}}$ ($i \in [r]$) and $w \in M^{\mu^{(k)}_r}$ and we will use the notation
\[
(v_1,\ldots,v_r; w) := v_1 \otimes \cdots \otimes v_r \otimes w.
\]
We define an $S_d \wr S_k$-action on $W$ by letting $(\bm \alpha,\beta) \in S_d \wr S_k$ act on such a vector as 
\begin{equation} \label{eq: action on W}
(\bm \alpha,\beta) (v_1,\ldots,v_r; w) = (\alpha_{\beta w(1)} v_1,\ldots, \alpha_{\beta w(r)} v_r; \beta w).
\end{equation}
Here, for ease of notation, we use the convention that for integers $0\leq r <k$, a tabloid $w$ of shape $\mu^{(k)}_r$ has content 
\begin{equation} \label{eq: def w}
w =\quad  \ytableausetup{centertableaux,boxsize=1.5em,tabloids} \begin{ytableau}
\mathbf i & \cdots & \mathbf j \\
w(1) \\
w(2) \\
\vdots \\
w(r)
\end{ytableau}.
\end{equation}

\begin{remark} \label{rem: our W is nice}
We emphasize the difference between the module $W$ and the construction of $S_d \wr S_k$-modules described in \cref{sec: rep wreath}: the action $(\bm \alpha; \beta)  \bigotimes_{i \in [k]} x_i = \bigotimes_{i \in [k]} \alpha_i  x_{\beta^{-1}(i)}$ described in \cref{eq: wreath action on boxtimes} permutes the tensor legs and then applies $\alpha_i$ to the $i$th leg, whereas the action described in \cref{eq: action on W} instead permutes the $\alpha_i$ but leaves the tensor legs fixed. To define the action on $W$ we require the ``L-shape'' of the tabloid $w$ from \eqref{eq: def w}. The fact that the tensor legs are not permuted by the $S_d \wr S_k$-action on $W$ will be crucial in deriving our decomposition. 
\end{remark}

\begin{lemma} \label{lem: VPQ using tabloids}
We have $V_{P,\bm Q} \cong W$ as $S_d \wr S_k$-modules. 
\end{lemma}
\begin{proof}
We construct an explicit $S_d \wr S_k$-isomorphism. Let $\phi: V_{P,\bm Q} \to W$ be the map that sends a monomial $x_{i_1,j_1} \cdots x_{i_t,j_t}$ in $V_{P,\mathbf Q}$ to the standard basis vector $(v_1,\ldots, v_r; w)$ where 
$w \in M^{\mu^{(k)}_r}$ is the tabloid for which $w(i) \in [k]$ is the label of the basis assigned to the monomials corresponding to $P_i$, and similarly $v_i(j) \in [d]$ is the basis element assigned to the monomials corresponding to the $j$th set in the set partition $Q_i$. The map $\phi$ is a bijection between the standard basis elements of $V_{P,\bm Q}$ and the standard basis elements of $W$ and we extend it linearly to all of $V_{P,\bm Q}$. It remains to show that $\phi$ is $S_d \wr S_k$-equivariant. Let $x_{i_1,j_1} \cdots x_{i_t,j_t} \in V_{P,\bm Q}$, define $(v_1,\ldots, v_r; w)= \phi(x_{i_1,j_1} \cdots x_{i_t,j_t})$, and let $(\bm \alpha, \beta) \in S_d \wr S_k$. We show that $\phi((\bm \alpha,\beta) (x_{i_1,j_1} \cdots x_{i_t,j_t}))$ agrees with \cref{eq: action on W}. Recall that 
\[
(\bm \alpha,\beta) (x_{i_1,j_1} \cdots x_{i_t,j_t}) = x_{\alpha_{\beta(j_1)}(i_1),\beta(j_1)} \cdots x_{\alpha_{\beta(j_t)}(i_t),\beta(j_t)}
\]
and therefore the assignment of bases and basis elements after the action of $(\bm \alpha,\beta)$ is as follows: the monomials corresponding to $P_i$ are assigned basis $\beta(w(i))$ and the monomials corresponding to the $j$th set in the set partition $Q_i$ are assigned basis element $\alpha_{\beta(w(i))} v_i(j)$. This shows that indeed 
\begin{align*}
\phi((\bm \alpha,\beta) (x_{i_1,j_1} \cdots x_{i_t,j_t})) &= (\alpha_{\beta w(1)} v_1,\ldots, \alpha_{\beta w(r)} v_r; \beta w) \\
&= (\bm \alpha, \beta) (v_1,\ldots,v_r; w) = (\bm \alpha,\beta) \phi(x_{i_1,j_1} \cdots x_{i_t,j_t}). \qedhere
\end{align*}
\end{proof}
We now use known results from the representation theory of the symmetric group to obtain a decomposition of $W$. Recall that $\nu_1,\ldots, \nu_\ell \vdash d$ is a complete list of partitions of $d$, sorted lexicographically, and hence that $\ell$ is the number of partitions of $d$. For $\mathbf j \in [\ell]^r$, let 
\begin{align} \label{eq: Tdjq}
\mathbf T^d_{\mathbf j,\mathbf q} = \bigtimes_{i \in [r]} T_{\nu_{j(i)},\mu^{(d)}_{q_i}}.
\end{align}
For each $\bm \sigma \in \bm T^d_{\bm j,\bm q}$ we define an $S_d \wr S_k$-submodule of $W$:
\[
W_{\bm \sigma} := (\sigma_1 \cdot S^{\nu_{j(1)}} \otimes \ldots \otimes \sigma_r \cdot S^{\nu_{j(r)}}) \otimes  M^{\mu^{(k)}_r}.
\]
Applying \cref{eq: decomp permutation module} to each of the first $r$ coordinates of $W$ shows the following.
\begin{lemma} \label{lem: first decomposition of W}
We have $W = \bigoplus_{\bm j \in [\ell]^r} \bigoplus_{\bm \sigma \in \bm T^d_{\bm j,\bm q}} W_{\bm \sigma}.$ 
\end{lemma}
We proceed by viewing the above module $W_{\bm \sigma}$ as an induced module, where the induction is from a group that corresponds to $\bm j$. The $r$-tuple $\bm j$ naturally gives rise to a Young subgroup of $S_r$. We would like to associate a subgroup of $S_k$ to it. We do this by first extending $\bm j$ to a $k$-tuple $\bm{\tilde j} = (j(1),\ldots,j(r),\underbrace{1,\ldots,1}_{k-r \text{ times}}) \in [\ell]^k$ and then associating a Young subgroup to it. Concretely, to the $r$-tuple $\bm j$ we associate the following Young subgroup of $S_k$: let $\bm{\tilde j} = (j(1),\ldots,j(r),\underbrace{1,\ldots,1}_{k-r \text{ times}}) \in [\ell]^k$ and define\footnote{The choice of $1$'s on the last $k-r$ coordinates might seem arbitrary, but it is motivated by the fact that $\nu_1 = (d)$ and $S^{(d)}$ is the trivial representation of $S_d$.} 
\begin{align}\label{eq: Sj}
S_{\bm j} := S_{\bm{\tilde j}} = S_{ \bm j^{-1}(1) \cup \{k-r+1,\ldots,k\}} \times S_{\bm j^{-1}(2)} \times \cdots \times S_{\bm j^{-1}(\ell)}.
\end{align}
We define partitions $\gamma^a$ for $a \in [\ell]$ as
\begin{align} \label{eq: gamma}
\gamma^a = \begin{cases}
(k-r,1^{|\bm j^{-1}(1)|}) \quad\quad\quad&\mbox{if } a = 1, \\
(1^{|\bm j^{-1}(a)|})  &\mbox{if } a \geq 2. 
\end{cases}
\end{align}
We next consider the permutation modules $M^{\gamma^a}$. For $a \geq 2$ we index the rows of tabloids in $M^{\gamma^a}$ with $\bm j^{-1}(a)$ and we let these tabloids have content $\bm j^{-1}(a)$. For $M^{\gamma^1}$ we index the last $|\bm j^{-1}(1)|$ rows using $\bm j^{-1}(1)$ and we let tabloids in $M^{\gamma^1}$ have content $\bm{\tilde j}^{-1}(1)$. 
We then define the $S_d \wr S_{\bm j}$-module 
\[
(\sigma_1 \cdot S^{\nu_{j(1)}} \otimes \ldots \otimes \sigma_r \cdot S^{\nu_{j(r)}}) \otimes (M^{\gamma^1} \otimes \ldots \otimes M^{\gamma^\ell}) 
\]
where the action of $(\bm \alpha,\beta) \in S_d \wr S_{\bm j}$ is defined on a standard basis vector 
$(v_1,\ldots, v_r;w_1,\ldots, w_{\ell})$ as follows
\[
(\bm \alpha,\beta) (v_1,\ldots, v_r;w_1,\ldots, w_{\ell}) = (\alpha_{\beta w (1)} v_1,\ldots, \alpha_{\beta w(r)} v_r;\beta w_1,\ldots, \beta w_{\ell})
\]
where $w:[r]\to [k]$ is such that the row of $w_1,\ldots, w_\ell$ that is indexed with $i \in [r]$ has content $w(i)$.
This $S_d \wr S_{\bm j}$-module is naturally isomorphic to a tensor product of an $S_d \wr S_{\bm{\tilde j}^{-1}(1)}$-module and $S_d \wr S_{\bm j^{-1}(a)}$-modules (for $a \geq 2$):
\begin{equation} \label{eq: reordering the coordinates}
(\sigma_1 \cdot S^{\nu_{j(1)}} \otimes \ldots \otimes \sigma_r \cdot S^{\nu_{j(r)}}) \otimes (M^{\gamma^1} \otimes \ldots \otimes M^{\gamma^{\ell}}) \cong \bigotimes_{a \in [\ell]} \left(\left(\bigotimes_{i \in \bm j^{-1}(a)} \sigma_i \cdot S^{\nu_a}\right) \otimes M^{\gamma^a}\right) 
\end{equation}

\begin{lemma} 
\label{lem: mu r to gammas} 
For $\bm j \in [\ell]^r$ and $\bm \sigma \in \bm T^d_{\bm j, \bm q}$, as $S_d \wr S_k$-modules, we have 
\[
W_{\bm \sigma} \cong \left[ (\sigma_1 \cdot S^{\nu_{j(1)}} \otimes \ldots \otimes \sigma_r \cdot S^{\nu_{j(r)}}) \otimes (M^{\gamma^1} \otimes \ldots \otimes M^{\gamma^{\ell}}) \right]\left.\vphantom{\big]}\right\uparrow_{d \wr \bm{j}}^{d \wr k} =: W_{\bm \sigma}'
\]
\end{lemma}
\begin{proof} We construct an explicit $S_d \wr S_k$-isomorphism $\phi: W_{ \bm \sigma} \to W_{ \bm \sigma}'$.
For each standard basis element~$(v_1,\ldots,v_r; w) \in W_{ \bm \sigma}$ we construct an element
$$
\phi(v_1,\ldots,v_r; w) = ((\id, b), (v_1, \ldots, v_r ; w_1, \ldots,  w_\ell)),
$$
where~$(\id,b)$ is a coset representative of~$S_d \wr S_{\bm j} $ in $S_d \wr S_k$. As coset representatives we take the elements~$(\id, b)$ for which $b \in S_k$ is such that $b(i)>b(i')$ whenever $i>i'$ and either $i,i' \in \bm{\tilde j}^{-1}(1)$ or there exists an $a \in [\ell]_{\geq 2}$ such that $i,i' \in \bm{j}^{-1}(a)$. The coset representative $b$ and $\ell$-tuple of tabloids $( w_1,\ldots, w_{\ell})$ where $ w_a \in M^{\gamma^a}$ ($a \in [\ell]$) are uniquely defined by the condition 
 \begin{equation} \label{eq: def coset elements}
 w(i) = b(w_{j(i)}(i)) \qquad \text{ for } i \in [r].
 \end{equation}
 This defines a bijection between the basis elements of $W_{ \bm \sigma}$ and those of $W_{ \bm \sigma}'$, and we extend $\phi$ linearly to all of $W_{ \bm \sigma}$.  It remains to show that~$\phi$ respects the group action. 
 
 Let $(\bm \alpha, \beta) \in S_d \wr S_k$ and let $(v_1,\ldots,v_r; w) \in W_{ \bm \sigma}$. Let $(\id,\tilde b)$ be the coset representative for which $(\bm \alpha,\beta) (\id,b) \in (\id,\tilde b) \cdot S_d \wr S_{\bm{j}}$ and let $(\bm{\tilde \alpha},\tilde \beta)$ be such that
\begin{equation*} 
(\bm \alpha,\beta) (\id,b) = (\id,\tilde b) (\bm{\tilde \alpha}, \tilde \beta).
\end{equation*}
Note that $\bm{\tilde \alpha}$, $\tilde \beta$ and $\tilde b$ are such that $\tilde \beta = \tilde b^{-1} \beta b \in S_{\bm{j}}$ and $\alpha_i = \tilde \alpha_{\tilde b^{-1}(i)}$ for all $i \in [k]$. In particular, $\tilde \alpha_i = \alpha_{\tilde b(i)}$ for all $i \in [k]$. Then we have
\begin{align*}
    (\bm \alpha,\beta) \phi(v_1,\ldots,v_r; w) &=  (\bm \alpha,\beta) \big((\id, b), (v_1, \ldots, v_r ; w_1, \ldots, w_\ell)\big)\\
    &= \big((\id,\tilde b),(\bm{\tilde \alpha}, \tilde \beta) (v_1, \ldots, v_r ; w_1, \ldots, w_\ell)\big) \\
    &= \big((\id,\tilde b),(\tilde \alpha_{\tilde \beta u(1)} v_1, \ldots, \tilde \alpha_{\tilde \beta u(r)} v_r ; \tilde \beta w_1, \ldots, \tilde \beta w_\ell)\big) \\
    &= \big((\id,\tilde b),  (\alpha_{\tilde b \tilde \beta u(1)} v_1, \ldots, \alpha_{\tilde b \tilde \beta u(r)} v_r ; \tilde \beta w_1, \ldots, \tilde \beta w_\ell) \big) \\
    &= \big((\id,\tilde b), (\alpha_{\beta b u(1)} v_1, \ldots, \alpha_{\beta  b u(r)} v_r ; \tilde \beta w_1, \ldots, \tilde \beta w_\ell)\big)
\end{align*}
where $u:[r]\to [k]$ is the function such that the row of $w_1,\ldots, w_\ell$ that is indexed with $i \in [r]$ has content $u(i)$. On the other hand, we have 
\begin{align*}
    \phi((\bm \alpha, \beta)(v_1,\ldots,v_r; w)) &= \phi(\alpha_{\beta w(1)} v_1,\ldots, \alpha_{\beta w(r)} v_r; \beta w) \\
    &= \big((\id,\tilde b), (\alpha_{\beta w(1)} v_1,\ldots, \alpha_{\beta w(r)} v_r; \tilde \beta w_1,\ldots, \tilde \beta w_\ell) \big) \\
    &= \big((\id,\tilde b), (\alpha_{\beta bu(1)} v_1,\ldots, \alpha_{\beta bu(r)} v_r; \tilde \beta w_1,\ldots, \tilde \beta w_\ell) \big) 
\end{align*}
where for the second to last step we use that
$\beta w(i) = \beta b(w_{j(i)}(i)) = \tilde b \tilde \beta (w_{j(i)}(i))$ for $i \in [r]$,
and for the last step we use \cref{eq: def coset elements}. \end{proof}

We then obtain our final decomposition by decomposing the permutation modules $M^{\gamma^a}$ as in \cref{eq: decomp permutation module} to obtain $S_d \wr S_k$-modules $W_{\bm \sigma, \bm \tau}$: let $\underline \Lambda$ be an $\ell$-multipartition of $|\underline \gamma|$ and define 
\begin{align} \label{eq: Tklg}
\bm{T}^k_{\underline \Lambda,\underline \gamma} := \bigtimes_{a \in [\ell]} T_{\Lambda^a,\gamma^a}. 
\end{align}
For each $\bm \sigma \in \bm T^d_{\bm j, \bm q}$ and $\bm \tau \in \bm T^k_{\underline \Lambda,\underline \gamma}$ we define an $S_d \wr S_k$-submodule of $W_{\bm \sigma}'$:
\[
W_{\bm \sigma,\bm \tau} := \left[ (\sigma_1 \cdot S^{\nu_{j(1)}} \otimes \ldots \otimes \sigma_r \cdot S^{\nu_{j(r)}}) \otimes (\tau_1 \cdot S^{\Lambda^1} \otimes \ldots \otimes \tau_\ell \cdot S^{\Lambda^\ell}) \right]\left.\vphantom{\Big]}\right\uparrow_{d \wr \bm{j}}^{d \wr k}
\]
\begin{lemma} \label{lem: second decomposition of W}
We have $\displaystyle{W_{\bm \sigma}' = \bigoplus_{\underline \Lambda \vdash |\underline \gamma|} \bigoplus_{\bm \tau \in \bm T^k_{\underline \Lambda,\underline \gamma}} W_{\bm \sigma,\bm \tau}}$.
\end{lemma}
\begin{proof}
This follows by applying \cref{eq: decomp permutation module} to each of the modules $M^{\gamma^a}$. 
\end{proof}
Combining \cref{lem: first decomposition of W,lem: second decomposition of W} this shows the following:
\begin{corollary} \label{lem: decomp of VPQ}
We have $\displaystyle{W \cong \bigoplus_{\bm j \in [\ell]^r} \bigoplus_{\bm \sigma \in \bm T^d_{\bm j,\bm q}} \bigoplus_{\underline \Lambda \vdash |\underline \gamma|} \bigoplus_{\bm \tau \in \bm T^k_{\underline \Lambda,\underline \gamma}} W_{\bm \sigma,\bm \tau}}$.
\end{corollary}

In the next section we provide the necessary lemmas that allow us to show this is in fact a decomposition of $W$ into \emph{irreducible} submodules. 
\begin{proposition} \label{prop: W sigma tau Specht}
For $\bm j\in [\ell]^r$, $\bm \sigma \in \bm T^d_{\bm j, \bm q}$, $\bm \tau \in \bm T^k_{\underline \Lambda, \underline \gamma}$, we have $\displaystyle{W_{\bm \sigma,\bm \tau} \cong S^{\underline \Lambda}_{\bm j} \cong S^{\underline \Lambda}}$ as $S_d \wr S_k$-modules.
\end{proposition} 
\begin{proof}
Recall that in \cref{lem: with equals without j} we have already shown that $S^{\underline \Lambda}_{\bm j}$ is $S_d \wr S_k$-isomorphic to $S^{\underline \Lambda}$. It thus remains to show that $W_{\bm \sigma,\bm \tau}$ is $S_d \wr S_k$-isomorphic to $S^{\underline \Lambda}_{\bm j}$. Since induction preserves isomorphism, it suffices to show that the respective $S_d\wr S_{\bm j}$-modules are $S_d \wr S_{\bm j}$-isomorphic. That is, after grouping on both sides the $\bm j^{-1}(a)$ coordinates (cf.~\eqref{eq: reordering the coordinates}), it suffices to show
\[
\bigotimes_{a \in [\ell]} \left(\left(\bigotimes_{i \in \bm j^{-1}(a)} \sigma_i \cdot S^{\nu_a}\right) \otimes \tau_a \cdot S^{\Lambda^a}\right) \cong \bigotimes_{a \in [\ell]}\left( \left(S^{\nu_a}\right)^{\tilde \boxtimes |\gamma^a|} \oslash  S^{\Lambda^a}\right)
\]
as $S_d \wr S_{\bm j}$-modules. We do so separately for each $a \in [\ell]$. For $a \geq 2$, we have that $\tau_a \cdot S^{\Lambda^a} \subseteq M^{(1^{|\gamma^a|})}$ and the $S_d \wr S_{\bm j^{-1}(a)}$-isomorphism follows from \cref{lem: from stationary to permuting,lem: from stationary to permuting 2}. For $a=1$ and thus $\nu_1 = (d)$, we have $\tau_1 \cdot S^{\Lambda^1} \subseteq M^{(k-r,1^{|\bm j^{-1}(1)|})}$ and the $S_d \wr S_{\bm j}$-isomorphism follows from \cref{lem: the case nu=(d)}.
\end{proof}

To obtain a representative set for the action of $S_d \wr S_k$ on $W$ it will be useful to group the terms in the decomposition of $W$ from \cref{lem: decomp of VPQ}  per Specht module $S^{\underline \Lambda}$. That is, we rearrange the direct sum as follows
\[
W \cong 
\bigoplus_{\underline \Lambda \vdash k} \left(\bigoplus_{\substack{\bm j \in [\ell]^r: \\
|\bm j^{-1}(a)| = |\Lambda^a|  \text{ for } a \geq 2}}
\bigoplus_{\bm \tau \in \bm T^k_{\underline \Lambda,\underline \gamma}} \bigoplus_{\bm \sigma \in \bm T^d_{\bm j,\bm q}}   W_{\bm \sigma,\bm \tau}\right).
\]

\subsection{\texorpdfstring{$W_{\bm \sigma,\bm \tau}$}{W sigma tau} is isomorphic to a Specht module} \label{sec: iso to Specht modules}

We first consider the setting where all $S_d$-modules are isomorphic to $S^\nu$ for a single $\nu \vdash d$ and on the last coordinate we have the module $M^{(1^k)}$. That is, for $\sigma_i \in T_{\nu,\mu_i}$ (for arbitrary $\mu_i \vdash d$) we consider the $S_d \wr S_k$-module 
\[
M_1 = \left(\bigotimes_{i \in [k]} \sigma_i \cdot S^\nu\right) \otimes M^{(1^k)}
\]
with the action from \eqref{eq: action on W}. We show $M_1$ is isomorphic to the $S_d \wr S_k$-module $M_2 = (S^\nu)^{\tilde \boxtimes r} \oslash M^{(1^k)}$.
\begin{lemma} \label{lem: first step towards Specht module}
As $S_d \wr S_k$-modules we have $M_1 \cong M_2$.
\end{lemma}
\begin{proof}
We construct an $S_d \wr S_k$-isomorphism $\phi: M_1 \to M_2$. As a standard basis of $M_1$ we consider the vectors of the form $(\sigma_1 \cdot v_1,\ldots, \sigma_k \cdot v_k; w) \in M_1$ where $v_1,\ldots, v_k$ are standard basis vectors of $S^\nu$ and $w \in M^{(1^k)}$ is a tabloid (this forms a basis since the $\sigma_i$ provide bijections between $S^\nu$ and $\sigma_i \cdot S^\nu$). For such a basis vector we define 
\[
\phi (\sigma_1 \cdot v_1,\ldots, \sigma_k \cdot v_k; w) = (v_{w^{-1}(1)},\ldots,v_{w^{-1}(k)}; w).
\]
This defines a bijection between the standard basis of $M_1$ and the standard basis of $M_2$. We extend~$\phi$ linearly to the vector space $M_1$. 
Recall that for $(\bm \alpha, \beta) \in S_d \wr S_k$ we have 
\begin{align*}
(\bm \alpha, \beta) (\sigma_1 \cdot v_1,\ldots, \sigma_k \cdot v_k;w) &= (\alpha_{\beta w(1)} \sigma_1 \cdot v_1,\ldots, \alpha_{\beta w(k)} \sigma_k \cdot v_k; \beta w) \\
&= (\sigma_1 \cdot \alpha_{\beta w(1)}  v_1,\ldots, \sigma_k \cdot\alpha_{\beta w(k)} v_k; \beta w), 
\end{align*}
and hence 
\begin{align*}
\phi\big((\bm \alpha, \beta) (\sigma_1 \cdot v_1,\ldots, \sigma_k \cdot v_k;w)\big) &= (\alpha_{\beta w((\beta w)^{-1}(1))} v_{(\beta w)^{-1}(1)},\ldots,\alpha_{\beta w((\beta w)^{-1}(k))}v_{(\beta w)^{-1}(k)}; \beta w ) \\
&= (\alpha_{1} v_{(\beta w)^{-1}(1)},\ldots,\alpha_{k}v_{(\beta w)^{-1}(k)}; \beta w ).
\end{align*}
On the other hand 
\begin{align*}
(\bm \alpha, \beta) \phi(\sigma_1 \cdot v_1,\ldots,\sigma_k \cdot v_k;w) &= (\bm \alpha,\beta) (v_{w^{-1}(1)},\ldots,v_{w^{-1}(k)}; w) \\
&= (\alpha_1 v_{w^{-1}\beta^{-1}(1)},\ldots,v_{w^{-1}\beta^{-1}(k)};\beta w).
\end{align*}
In other words, $\phi$ is $S_d \wr S_k$-equivariant. 
\end{proof}

We now restrict the isomorphism from the above lemma to an $S_d \wr S_k$-submodule to obtain:
\begin{lemma} \label{lem: from stationary to permuting}
For $i \in [k]$, let $\lambda_i \vdash k$ and $\sigma_i \in T_{\nu,\lambda_i}$. For $\Lambda \vdash k$ and $\tau \in T_{\Lambda,(1^k)}$, we have 
\[
\Big(\bigotimes_{i \in [k]} \sigma_i \cdot S^\nu\Big) \otimes \tau\cdot S^{\Lambda} \cong (S^\nu)^{\tilde \boxtimes k} \oslash \tau \cdot S^{\Lambda}.
\]
\end{lemma}
\begin{proof}
Let $\phi$ be the $S_d\wr S_k$-isomorphism from the proof of \cref{lem: first step towards Specht module}. By restricting $\phi$ to $\left(\bigotimes_{i \in [k]} \sigma_i \cdot S^\nu\right) \otimes \tau\cdot S^{\Lambda}$ (an $S_d \wr S_k$-submodule of its domain) we obtain an $S_d \wr S_k$-isomorphism between $\left(\bigotimes_{i \in [k]} \sigma_i \cdot S^\nu\right) \otimes \tau\cdot S^{\Lambda}$ and $\phi\left(\left(\bigotimes_{i \in [k]} \sigma_i \cdot S^\nu\right) \otimes \tau\cdot S^{\Lambda}\right)$. It remains to show that the image is in fact $(S^\nu)^{\tilde \boxtimes k} \oslash \tau \cdot S^{\Lambda}$. Let $e_\nu \in S^\nu$ and $e_\Lambda \in S^\Lambda$ be the respective generating elements. We first show that $\phi(e_\nu,\ldots, e_\nu; e_\Lambda) \in (S^\nu)^{\tilde \boxtimes k} \oslash \tau \cdot S^{\Lambda}$. 
Recall that 
$\tau \cdot e_\Lambda = \sum_{\tau' \sim \tau} \sum_{c \in C_{t_\Lambda}} \sign(c) \, \tau' *  c \{t_\Lambda\}$. 
Using the linearity of $\phi$, we have 
\begin{align*}
&\phi(\sigma_1 \cdot e_\nu,\ldots,\sigma_k \cdot e_\nu;\tau \cdot e_\Lambda) \\
&= \sum_{\tau' \sim \tau} \sum_{c \in C_{t_\Lambda}} \sign(c) \, \phi(\sigma_1 \cdot e_\nu,\ldots,\sigma_k \cdot e_\nu;\tau' *  c  \{t_\Lambda\}) = \sum_{\tau' \sim \tau} \sum_{c \in C_{t_\Lambda}} \sign(c) \, (e_\nu,\ldots, e_\nu;\tau' *  c  \{t_\Lambda\}) \\
&= (e_\nu,\ldots, e_\nu; \sum_{\tau' \sim \tau} \sum_{c \in C_{t_\Lambda}} \sign(c) \, \tau' *  c  \{t_\Lambda\}) = (e_\nu,\ldots, e_\nu; \tau \cdot e_\Lambda) \in (S^\nu)^{\tilde \boxtimes k} \oslash \tau \cdot S^{\Lambda}.
\end{align*}
In other words, the generating element of $\left(\bigotimes_{i \in [k]} \sigma_i \cdot S^\nu\right) \otimes \tau\cdot S^{\Lambda}$ is mapped to the generating element of $(S^\nu)^{\tilde \boxtimes k} \oslash \tau \cdot S^{\Lambda}$. Since $\phi$ is $S_d \wr S_k$-equivariant, this shows that 
\[
\phi\Big(\Big(\bigotimes_{i \in [k]} \sigma_i \cdot S^\nu\Big) \otimes \tau\cdot S^{\Lambda}\Big) = (S^\nu)^{\tilde \boxtimes k} \oslash \tau \cdot S^{\Lambda}. \qedhere
\]
\end{proof}

Finally, we show that $\tau$ can be removed, by using~$\mu = (1^k)$ in the proposition below. 
\begin{lemma} \label{lem: from stationary to permuting 2}
For $\Lambda, \mu \vdash k$ and $\tau \in T_{\Lambda,\mu}$, we have $(S^\nu)^{\tilde \boxtimes k} \oslash \tau \cdot S^{\Lambda}\cong(S^\nu)^{\tilde \boxtimes k} \oslash S^{\Lambda}$.
\end{lemma}
\begin{proof}
Let $v_1,\ldots, v_k$ be basis elements of $S^\nu$ and let $w$ be a basis vector of $S^\Lambda$. We define 
$\phi(v_1,\ldots,v_k;\tau \cdot w) = (v_1,\ldots,v_k;w)$.
Since $\tau$ forms a linear bijection between $S^\Lambda$ and $\tau \cdot S^\Lambda$, this provides a linear bijection between $ (S^\nu)^{\tilde \boxtimes k} \oslash \tau \cdot S^{\Lambda}$ and $(S^\nu)^{\tilde \boxtimes k} \oslash S^{\Lambda}$. Moreover, $\phi$ is $S_d \wr S_k$-equivariant since $\tau$ is $S_k$-equivariant.
\end{proof}

We finally consider the case $\nu = (d)$ and $\Lambda = (k-r,1^r)$ (for some $r<k$).  Recall that if~$M$ is an~$S_k$-module, it also is an~$S_d \wr S_k$-module via~$(\bm \alpha, \beta)m = \beta m$ for~$m \in M$ and~$(\bm \alpha, \beta) \in S_d \wr S_k$. By taking a tensor product with sufficiently many copies of the standard basis element $\ytableausetup{centertableaux,boxsize=1.0em,tabloids} \begin{ytableau}
1 & \cdots & d
\end{ytableau}$ of~$S^{(d)}$, one can show the following.
\begin{lemma} \label{lem: the case nu=(d)}
For $\Lambda \vdash k$ and $\tau \in T_{\Lambda,(k-r,1^r)}$, we have 
\begin{align*} 
  \Big(\bigotimes_{i \in [r]} \sigma_i \cdot S^{(d)}\Big) \otimes \tau\cdot S^{\Lambda} \cong   (S^{(d)})^{\tilde \boxtimes k} \oslash \tau \cdot S^{\Lambda} \cong  (S^{(d)})^{\tilde \boxtimes k} \oslash S^{\Lambda}  \cong S^{\Lambda}.
\end{align*}
\end{lemma}

\subsection{Decomposing \texorpdfstring{$\C^{([d]\times [k])^t}$}{Cdkt} into Specht modules} \label{sec: explicit decomposition}

Combining the results from the previous sections, we obtain a decomposition of $\C^{([d]\times [k])^t}$ into irreducible modules. 
We have 
\begin{align} \label{eq: finaldecomp}
    \C^{([d] \times [k])^t} 
    &= \bigoplus_{\substack{(P,\mathbf Q) \in ([d]\times[k])^t / (S_d \wr S_k)}} \phi_{P,\mathbf Q}  \bigoplus_{\underline \Lambda \vdash k} \left(\bigoplus_{\substack{\bm j \in [\ell]^r: \\
|\bm j^{-1}(a)| = |\Lambda^a|  \text{ for } a \geq 2}}
\bigoplus_{\bm \tau \in \bm T^k_{\underline \Lambda,\underline \gamma}} \bigoplus_{\bm \sigma \in \bm T^d_{\bm j,\bm q}}   W_{\bm \sigma,\bm \tau}\right) 
\end{align}
where $\phi_{P, \bm Q}$ is the isomorphism provided in \cref{lem: VPQ using tabloids}, $r = |P|$, $q_i = |Q_i|$. 

From this decomposition we derive an explicit representative set by selecting for each $W_{\bm \sigma, \bm \tau}$ the canonical generating element as the representative element (and grouping them according to $\underline \Lambda$). The representative vectors~$u_{(\bm \sigma, \bm \tau),(P,{\bm Q})}$ will be constructed from pairs~$(\bm \sigma, \bm \tau)$ and $(P,{\bm Q})$ appearing in the above decomposition. We now give precise definitions.

\paragraph{Representative elements.}
Fix a pair~$(P,{\bm{Q}})$ where~$P=\{P_1,\ldots,P_r\}$ is a set partition of~$[t]$ into at most $k$ parts and~${\bm{Q}} = (Q_1,\ldots,Q_r)$, where each~$Q_i$ is a set partition of~$P_i$ in at most~$ d$ parts. Let $\bm{j} \in [\ell]^r$, $\bm q = (q_1,\ldots,q_r) = (|Q_1|,\ldots,|Q_r|)$, $\underline \Lambda  \vdash |\underline{\gamma}|$ where $\underline \gamma$ is defined in \cref{eq: gamma}, and~$\bm \sigma \in \bm T^d_{\mathbf j,\mathbf q}$ and $\bm \tau \in \bm{T}^k_{\underline\Lambda,\underline\gamma}$. The representative vector corresponding to $(\bm \sigma, \bm \tau)$ and~$(P,{\bm Q})$ is equal to 
\begin{align*}
u_{(\bm \sigma, \bm \tau),(P,{\bm Q})} := \phi^{-1}_{P,{\bm Q}} \big((\id, \id), (\sigma_1 \cdot v_1,\ldots,\sigma_r \cdot v_r; \tau_1 \cdot u_1,\ldots, \tau_\ell \cdot u_\ell)\big),
\end{align*}
where the vector $\big((\id,\id),(v;u)\big)$ is the generating element of $S^{\underline\Lambda}_{\bm j}$  described in~Section~\ref{sec: rep wreath}. Explicitly, we have
\begin{align*} 
u_{(\bm \sigma, \bm \tau),(P,{\bm Q})} = \phi^{-1}_{P,{\bm Q}} \left( \left( \bigotimes_{i=1}^r \sigma_i \cdot e_{\nu_{j(i)}} \right) \otimes \left(\bigotimes_{a=1}^\ell \tau_a \cdot e_{\Lambda^a} \right) \right). 
\end{align*} 
The next step is to describe these vectors explicitly as vectors in~$\C^{([d] \times [k])^t}$.
Note that
$$
 V_{Q_1} \otimes \cdots \otimes V_{Q_r} \otimes V_P \cong  V_{P,\bm{Q}}  \subseteq \C^{([d] \times [k])^t},
$$
via the isomorphism given by
$w_1\otimes \cdots \otimes w_r \otimes u \mapsto \text{Perm}_P(w_1\otimes \cdots \otimes w_r) \otimes u$.
Here~$\text{Perm}_P$ is the permutation matrix which permutes the coordinates of~$w_1\otimes \cdots \otimes w_r$ (which is a vector in~$\C^{[d]^t}$) in such a way that~$w_i$ corresponds to the coordinates given by~$P_i$. Then
\begin{align*} 
\phi_{P,\bm{Q}}^{-1}((\id, \id), (v_1,\ldots,v_r;u_1,\ldots,u_{\ell})) = \text{Perm}_P \big(\phi_{Q_1}^{-1}(v_1)\otimes \ldots \otimes \phi_{Q_r}^{-1}(v_r)\big) \otimes \phi_{P}^{-1}(u_1\otimes \cdots \otimes u_{\ell}),
\end{align*}
where~$\phi_P : V_P \to M^{\mu_r}$ is the isomorphism from \cref{reprsetSk} and we identify \smash{$\left.\left[\bigotimes_{a \in [\ell]} M^{\gamma^a}\right]\right\uparrow_{d \wr \bm j}^{d \wr k}$} with $M^{\mu^{(k)}_r}$ (cf.~\cref{lem: mu r to gammas}). 

\paragraph{The dimension of \texorpdfstring{$\mathrm{End}_{S_d \wr S_k}(\C^{([d] \times [k])^t})$}{Sd-wreath-Sk-invariant endomorphisms on V}.}
As in \cref{sec: Sk-invariant endos}, we note that $\mathrm{dim}(\mathrm{End}_{S_d \wr S_k}(\C^{([d] \times [k])^t}))$ equals the dimension of the space of $S_d \wr S_k$-invariant noncommutative polynomials of degree exactly~$2t$. In other words, it equals the number of pairs $(P,\{Q_1,\ldots,Q_r\})$ where~$P$ is a set partition of~$[2t]$ which determines the assignment of bases (to the~$k$-part) with~$|P|=r$, and~$Q_i$ is a set partition of~$P_i$ which determines the assignment of the basis elements (to the $d$-part). When $d,k \geq 2t$, this number is independent of $k$ and $d$ and simply becomes the number of pairs of set partitions~$(P,P')$ of  a set of  $2t$ elements with the property that the set partition~$P'$ is a refinement of~$P$. The sequence of numbers $(C_n)_{n \geq 0}$ such that~$C_n$ is the number of pairs of set partitions~$(P,P')$ with $P$ a set partition of $[n]$ and $P'$ a refinement of~$P$ is known in the online encyclopedia of integer sequences as sequence~\texttt{A000258}; its first ten elements are 1, 1, 3, 12, 60, 358, 2471, 19302, 167894, 1606137. We emphasize that for $d,k \geq 2t$, the size of the symmetry reduced SDP thus becomes independent of $d$ and $k$.

\subsection{Representative set for the \texorpdfstring{$S_{d-1} \times (S_d \wr S_{k-1})$}{Sd-1 times (Sd wr Sk-1)} action on \texorpdfstring{$(1,1) \times ([d] \times [k])^t$}{(1,1) times (d times k) power t }}
\label{sec: representative set thalf}
We will now give a representative set useful for the case~$t+\tfrac 12$, as discussed in \cref{sec: groupinvariance}. 
We naturally identify $x_{1,1}\R\langle x_{i,j}:i \in [d], j \in [k]\rangle_{=t}$ with $\C^{(1,1) \times ([d] \times [k])^t}$. Recall that this set is a module for~$G:=S_{d-1} \times (S_d \wr S_{k-1})$, where~$S_{d-1}$ permutes the indices $2,\ldots,d$ corresponding to~$k=1$, and where~$S_d \wr S_{k-1}$ acts on the indices~$(i,j)$ with $ i \in [d], \, j \in \{2,\ldots, k\}$. 

As before, we start by decomposing $\C^{(1,1) \times ([d] \times [k])^t}$ using the equivalence classes $((1,1) \times([d]\times[k])^t) / (S_{d-1} \times (S_d \wr S_{k-1}))$. Such an equivalence class corresponds to a pair $(P,\bm Q)$ where $P = \{P_1,\ldots, P_r\}$ is a set partition of $[t+1]$ in at most $k$ parts and $\mathbf Q = (Q_1,\ldots,Q_r)$ where each $Q_i$ is a set partition of $P_i$ in at most $d$ parts. This gives a first decomposition of our vector space:  
\begin{equation*}
\C^{(1,1) \times ([d] \times [k])^t} = \bigoplus_{\substack{(P,\mathbf Q) \in ((1,1) \times([d]\times[k])^t) / (S_{d-1} \times (S_d \wr S_{k-1}))}} V_{P,\mathbf Q} 
\end{equation*}
In the decomposition, we may assume that~$1 \in P_1$ (and hence gets assigned basis~$1$) and that $1 \in Q_1^{(1)}$ (and hence gets assigned basis element~$1$). 
We then have that $V_{P,\bm Q}$ is $S_{d-1} \times (S_d \wr S_{k-1})$-isomorphic to the tensor product~$V_1 \otimes V_2$, where
\begin{align*}
    V_1 := M^{\mu^{(d-1)}_{q_1-1}}, \qquad V_2 := (M^{\mu^{(d)}_{q_2}} \otimes \ldots \otimes M^{\mu^{(d)}_{q_r}}) \otimes  M^{\mu^{(k)}_{r-1}}. 
\end{align*}
The group~$G_1:=S_{d-1}$ acts on~$V_1$, and $G_2:=S_d \wr S_{k-1}$ acts on~$V_2$. To obtain a representative set for the action of~$G_1 \times G_2$ on~$V_1 \otimes V_2$, we use \cref{eq: prodrep} to combine representative sets for the actions of~$G_1$ and~$G_2$ on~$V_1$ and~$V_2$ from Sections~\ref{sec:reprset sk} and~$\ref{sec:reprset wreath}$, respectively. One then obtains a representative set for the $G$-action on $\C^{(1,1) \times ([d] \times [k])^t}$ by combining the representative sets of all $V_{P,\bm Q}$.

\subsection{An example: homogeneous quadratic polynomials}

\label{sec: example 232}
To illustrate the theory developed in Section~\ref{sec:reprset wreath}, we show how to compute one particular representative element. We consider homogeneous quadratic polynomials in variables~$x_{i,j}$, where~$i\in [d]$ and~$j \in [k] $, that are invariant under the action of $S_d \wr S_k$. 
We express the concepts in the language of polynomials. In this example we take~$d=2$,~$k=3$ and $t=2$. 

We first decompose the space of polynomials into $S_d \wr S_k$-invariant subspaces corresponding to orbits of monomials. In this example we consider the orbit
$$
\{x_{i_1,j_1}x_{i_2,j_2} \,\, | \,\, i_1,i_2 \in [d],\,\, j_1,j_2 \in [k] \text{ with }  j_1 \neq j_2\}.
$$
This orbit corresponds to the tuple of partitions~$(P, \bm Q) = (\{\{1\},\{2\}\},(\{\{1\}\},\{\{2\}\}))$. Here~$P$ is a partition of~$[t]$  and~$\bm Q$ is a tuple of set partitions~$\bm Q=(Q_1,\ldots,Q_{|P|})$ where each~$Q_i$ is a set partition of~$P_i$ in at most~$d$ parts. The corresponding subspace is
$$
V_{P, \bm Q } = \text{Span}\{x_{i_1,j_1}x_{i_2,j_2} \,\, | \,\, i_1,i_2 \in [d],\,\, j_1,j_2 \in [k] \text{ with }  j_1 \neq j_2\}.
$$

We have seen that, in order to decompose~$V_{P, \bm Q}$, we consider $\ell$-multipartitions of~$k=3$, where we recall that $\ell$ is the number of partitions of $d$. In this example $d=2$ and therefore $\ell=2$.  For this example we take the $\ell$-multipartition $\underline \Lambda = ((1),(1,1))$. The rows and columns of the block indexed by~$\underline \Lambda$ in the block diagonalization are indexed by tuples~$(\bm \sigma, \bm \tau),(P, \bm Q)$, where~$\bm \sigma$ is an element of $\bm T^d_{\bm j, \bm q}$ and~$\bm \tau$ is an element of $\bm{T}^k_{\underline \Lambda,\underline \gamma}$. 

To describe $\bm T^d_{\bm j, \bm q}$, we first need $\bm j$ and $\bm q$. Throughout this section we have~$r=|P|$, so in this case~$r=2$. In the final decomposition in Eq.~\eqref{eq: finaldecomp}, we have $\bm j \in [\ell]^r=[2]^2$, where~$|\bm j^{-1}(2)|=|\Lambda^2| = 2$, so~$\bm{j} = (2,2)$. Furthermore, we have~$\bm q = (|Q_1|,\ldots, |Q_r|)=(|Q_1|,|Q_2|)=(1,1)$. The collection $\bm T^d_{\bm j, \bm q}$ is a tuple of sets of semistandard tableaux (cf. Eq~\eqref{eq: Tdjq}), one set for each set in $P$ ($r=|P|$), where the content is defined by the size of the corresponding partition in $\bm Q$, and the shape is~$\nu_{\bm j(i)}$, for~$i=1,\ldots,r$, which is determined by~$\underline \Lambda$ as~$\bm j$ is determined by~$\underline \Lambda$. Formally, we have (cf.~Eq.~\eqref{eq: Tdjq})
$$
\bm T_{\bm j, \bm q}^d = \bm T_{(2,2), (1,1)}^2 =  \bigtimes_{i \in [r]} T_{\nu_{j(i)},\mu^{(d)}_{q_i}} =T_{(1,1),\mu^{(2)}_{1}} \times T_{(1,1),\mu^{(2)}_{1}} =  \left\{\ytableausetup{centertableaux,boxsize=1em,tabloids=off} \begin{ytableau}
1\\2
\end{ytableau}\right\} \times \left\{ \begin{ytableau}
1\\2\\
\end{ytableau}\right\},
$$ 
which consists of only one element: $\bm \sigma = \left(\ytableausetup{centertableaux,boxsize=1em,tabloids=off} \begin{ytableau}
1\\2
\end{ytableau}, \begin{ytableau}
1\\2\\
\end{ytableau}\right)$.

Similarly, to describe $\bm T_{\underline \Lambda, \underline\gamma}^k$, we first describe~$\underline \gamma$. The tuple~$\bm j$ is via Eq.~\eqref{eq: gamma} associated with the multipartition $\underline \gamma = ((k-r,0),(1,1))=((1),(1,1))$. The collection $\bm T_{\underline \Lambda, \underline\gamma}^k$ is a tuple of sets of semistandard tableaux (cf. Eq~\eqref{eq: Tklg}), one set for each partition in the multipartition~$\underline \Lambda$, where the content is defined by the size of the corresponding partition in $\underline \gamma$. Formally, we have (cf.~Eq.~\eqref{eq: Tklg})
$$
\bm{T}^k_{\underline \Lambda,\underline \gamma} = \bigtimes_{a \in [\ell]} T_{\Lambda^a,\gamma^a} = T_{(1),(1)} \times T_{(1,1),(1,1)} =  \left\{\ytableausetup{centertableaux,boxsize=1em,tabloids=off} \begin{ytableau}
1
\end{ytableau}\right\} \times \left\{ \begin{ytableau}
1\\2\\
\end{ytableau}\right\},
$$
which consists of one element, namely
$\bm \tau 
= \left(\ytableausetup{centertableaux,boxsize=1em,tabloids=off} \begin{ytableau}
1\\
\end{ytableau}, \begin{ytableau}
1\\2\\
\end{ytableau}\right)
$. We have now defined $(\bm \sigma, \bm \tau)$ and $(P,{\bm Q})$, so we are ready to compute the representative element~$u_{(\bm \sigma, \bm \tau),(P,{\bm Q})}$. We have
\begin{align*}
    u_{(\bm \sigma, \bm \tau),(P,{\bm Q})} &=  \phi^{-1}_{P,{\bm Q}} \left( \left( \bigotimes_{i=1}^r \sigma_i \cdot e_{\nu_{j(i)}} \right) \otimes \left(\bigotimes_{a=1}^\ell \tau_a \cdot e_{\Lambda^a} \right) \right)
    \\ &=  \phi^{-1}_{P,{\bm Q}} \left( \left( \left(\ytableausetup{tabloids} \begin{ytableau}
1\\2
\end{ytableau} -\begin{ytableau}
2\\1
\end{ytableau}\right)\otimes \left(\begin{ytableau}
1\\2
\end{ytableau} -\begin{ytableau}
2\\1
\end{ytableau}\right)\right) \otimes \left( (\begin{ytableau}
1
\end{ytableau})\otimes \left(\begin{ytableau}
1\\2
\end{ytableau} -\begin{ytableau}
2\\1
\end{ytableau}\right) \right) \right) \\&= \text{Perm}_P \left( \phi_{\{\{1\}\}}^{-1}\left(\ytableausetup{tabloids} \begin{ytableau}
1\\2
\end{ytableau} -\begin{ytableau}
2\\1
\end{ytableau}\right)\otimes \phi_{\{\{2\}\}}^{-1}\left(\begin{ytableau}
1\\2
\end{ytableau} -\begin{ytableau}
2\\1
\end{ytableau}\right)\right) \otimes \phi_{\{\{1\},\{2\}\}}^{-1} \left( (\begin{ytableau}
1
\end{ytableau})\otimes \left(\begin{ytableau}
1\\2
\end{ytableau} -\begin{ytableau}
2\\1
\end{ytableau}\right) \right)
\end{align*}
As in the discussion below Lemma~\ref{lem: first decomposition of W}, we identify \smash{$\left.\left[\bigotimes_{a \in [\ell]} M^{\gamma^a}\right]\right\uparrow_{d \wr \bm j}^{d \wr k}$} with $M^{\mu^{(k)}_r}$. In our case, $S_{\bm j }= S_{\{3\}} \times S_{\{1,2\}}$, cf.\ Eq.~\eqref{eq: Sj}, so the element 
$$(\begin{ytableau}
1
\end{ytableau})\otimes \left(\begin{ytableau}
1\\2
\end{ytableau} -\begin{ytableau}
2\\1
\end{ytableau}\right) \text{ is identified with } \begin{ytableau}
3\\1\\2
\end{ytableau} -\begin{ytableau}
3\\2\\1
\end{ytableau} \in M^{\mu_2^{(3)}}.
$$ Therefore we find
\begin{align*}
    u_{(\bm \sigma, \bm \tau),(P,{\bm Q})} &= \text{Perm}_P \left( \phi_{\{\{1\}\}}^{-1}\left(\ytableausetup{tabloids} \begin{ytableau}
1\\2
\end{ytableau} -\begin{ytableau}
2\\1
\end{ytableau}\right)\otimes \phi_{\{\{2\}\}}^{-1}\left(\begin{ytableau}
1\\2
\end{ytableau} -\begin{ytableau}
2\\1
\end{ytableau}\right)\right) \otimes \phi_{\{\{1\},\{2\}\}}^{-1} \left( \begin{ytableau}
3\\1\\2
\end{ytableau} -\begin{ytableau}
3\\2\\1
\end{ytableau}\right) 
\\&= \text{Perm}_P \underbrace{\left( (y_2-y_1)(y_2-y_1)\right)}_{\text{`$d$-part'}} \otimes \underbrace{(z_1z_2-z_2z_1)}_{\text{`$k$-part'}} 
\\&= x_{2,1}x_{2,2}-x_{1,1}x_{2,2}-x_{2,1}x_{1,2}+x_{1,1}x_{1,2} -x_{2,2}x_{2,1}+x_{1,2}x_{2,1}+x_{2,2}x_{1,1}-x_{1,2}x_{1,1}.
\end{align*}
In this case, Perm$_P$  can be viewed as identifying $x_{i,j} = y_i z_j$, where~$i \in [d]$ and~$j \in [k]$.
The other representative elements are computed in an analogous way, by considering different tuples of set partitions~$(P,\bm Q )$, $\ell$-multipartitions~$\underline \Lambda$, and tuples of Young tableaux~$\bm \sigma$ and $\bm \tau$.

\section{\texorpdfstring{$V_{P,\bm Q}$}{VPQ} is isomorphic to a permutation module \texorpdfstring{$M^{\underline{\mu}}$}{Mmu}.} \label{sec: iso to perm module}

Here we show that the $S_d \wr S_k$-modules $V_{P,\bm Q}$ from \cref{eq: first decomposition of V} can alternatively be viewed as a generalization of permutation modules to the wreath product setting. Our decomposition of $V_{P,\bm Q}$ can thus be viewed as an explicit decomposition of such a permutation module in Specht modules, or alternatively, as a characterization of the space of homomorphisms from Specht modules to such permutation modules. 

Analogously to \cref{def: Specht module wreath} for Specht modules, we use the following definition of a permutation module for $S_d \wr S_k$; these modules were named permutation modules e.g.~in \cite{Green19}.
\begin{definition} \label{def: permutation module}
For an $\ell$-multipartition $\underline \mu$ of $k$, the \emph{permutation module} $M^{\underline \mu}$ is defined as 
\[
M^{\underline \mu} := \Big[(M^{\nu_1}, \ldots, M^{\nu_{\ell}})^{\tilde \boxtimes |\underline \mu|} \oslash (M^{\mu^1} \boxtimes \cdots \boxtimes M^{\mu^\ell}) \Big]\left.\vphantom{\Big]}\right\uparrow_{d \wr |\underline \mu|}^{d \wr k}.
\]
\end{definition}
\noindent Standard basis elements of $M^{\underline \mu}$ are of the form 
$\big((\id,b),(v_1 \otimes \cdots \otimes v_k \otimes u_1 \otimes \cdots \otimes u_\ell)\big)$
where $(\id,b)$ is a coset representative, $v_1 \otimes \cdots \otimes v_k$ is an element of $(M^{\nu_1}, \ldots, M^{\nu_{\ell}})^{\tilde \boxtimes |\underline \mu|}$, and for each $a \in [\ell]$, $u_a$ is a tabloid in $M^{\mu^a}$ with content  $I_a := \{1+\sum_{j<a} |\mu^{j}|, \ldots,\sum_{j \leq a} |\mu^j|\}$ (if $|\mu^a|>0$). To make the distinction between $S_d$-modules and $S_k$-modules more apparent we will write $(v_1,\ldots,v_k; u_1,\ldots, u_k)$ to denote the vector $v_1 \otimes \cdots \otimes v_k \otimes u_1 \otimes \cdots \otimes u_\ell$. That is, we write standard basis vectors of $M^{\underline \mu}$ as $\big((\id,b),(v_1,\ldots,v_k; u_1,\ldots,u_\ell)\big)$.

\subsection{The isomorphism}
Given~$P$ and $\bm{Q} =(Q_1,\ldots,Q_r)$, set~$q_i:=|Q_i|$ for each~$i \in [r]$. Recall that in the vector space $V_{P,\bm Q}$ the set partition $P$ is used to assign \emph{distinct} bases ($j$'s) to variables in the monomial $x_{i_1,j_1} \cdots x_{i_t,j_t}$. The set partition $Q_i$ of $P_i$ is then used to assign $q_i = |Q_i|$ \emph{distinct} basis elements ($i$'s) to the variables in $P_i$. We now show how to interpret $V_{P,\bm Q}$ as a permutation module. 

For each $j \in [t-r]$, let $n_j = |\{i \in [r]: s_i = j\}|$. 
Define the $\ell$-multipartition $\underline \mu = (\mu^1,\ldots,\mu^\ell)$ by 
\begin{align*} 
\mu^a = \begin{cases}
(k-r) \quad\quad\quad&\mbox{if } a = 1 \\
( 1^{n_j} ) &\mbox{if }\nu_a = \mu^{(d)}_{j} \text{ for some $j \in [t-r]$} \\
(\,) & \mbox{otherwise}.
\end{cases}
\end{align*}
Note that $\sum_{j \in [t-r]} n_j = r$ and~$\mu^1=(k-r)$, so~$\underline{\mu}$ 
is an~$\ell$-multipartition of~$k$.
\begin{lemma}\label{Vpq isom Mgamma} We have $V_{P,\mathbf Q} \cong M^{\underline \mu}$.
\end{lemma}
\begin{proof}
For this proof it will be notationally convenient to change the order of the partitions of~$d$ in $M^{\underline \mu}$: we assume $\nu_j = \mu^{(d)}_j$ for $j \in [t-r]$ and we let $\nu_\ell = (d)$. Correspondingly, we then have $\mu^j = (1^{n_j})$ for $j \in [t-r]$ and $\mu^\ell=(k-r)$. Changing the order of the partitions of~$d$ leads to an $S_d\wr S_k$-isomorphic module, which, with abuse of notation, we will denote with $M^{\underline \mu}$ as well. We consider $S_d \wr S_{|\underline \mu|}$ as a subgroup of $S_d \wr S_k$. As coset representatives we take the elements $(\id,b)$ where $b \in S_k$ is such that $b(i)>b(i')$ whenever there exists a $a \in [\ell]$ such that $i,i' \in I_a$ and $i >i'$.

We first define a bijection between monomials in $V_{P,\mathbf Q}$ and basis vectors in $M^{\underline \mu}$. Recall that 
\[
I_j = I(\underline \mu)_j = \Big\{1+\sum_{i<j} |\mu^i|,\ldots,|\mu^j|+\sum_{i<j} |\mu^i|\Big\} = \Big\{1+\sum_{i<j} n_i,\ldots, n_j + \sum_{i<j} n_i\Big\}.
\]
Then, without loss of generality, we may label the $P_i$'s and corresponding $Q_i$'s in such a way that $|Q_i| = n_j$ if $i \in I_j$ (each $i \in [t]$ is contained in exactly one such $I_j$).  Given a monomial $x_{i_1,j_1} \cdots x_{i_t,j_t}$ in $V_{P,\mathbf Q}$, we define the following basis vector in $M^{\underline \mu}$. The assignment of bases to $P_1,\ldots,P_r$ uniquely determines a coset representative $b$ and an $\ell$-tuple $(u_1,\ldots, u_\ell)$ of tabloids where each $u_a:I_a\to I_a$ has shape $\mu^a$: for $j \in [t-r]$ the tabloid $u_j$ is such that if $i \in I_j$, then $P_i$ is assigned basis $b(u_j(i))$, and finally we let $u_\ell$ be the tabloid of shape $(k-r)$ with content $I_\ell$. 
For $i \in I_j$ define a tabloid $v_i$ of shape $\mu_j^{(d)}$ as \[
v_{i} = \qquad \ytableausetup{tabloids, boxsize=2em}\begin{ytableau}
\phantom{,}  & \cdots & \phantom{,} \\
w_{u_j^{-1}( i)}^1 \\
\vdots \\
w_{u_j^{-1}( i)}^{j} \\
\end{ytableau}
\]
where $w_{u_j^{-1}( i)}^{j'} \in [d]$ is the basis element assigned to the $j'$-th set in $Q_{u_j^{-1}(i)}$. We let $v_{r+1},\ldots,v_{k}$ equal the tabloid of shape $(d)$ with content $(1^d)$. 
Combining the above, to a monomial $x_{i_1,j_1} \cdots x_{i_t,j_t}$ in $V_{P,\mathbf Q}$, we associate the basis vector
\begin{equation} \label{eq: basis vector from monomial}
\big((\id,b),(v_1,\ldots,v_k; u_1,\ldots,u_{\ell})\big) \in M^{\underline \mu}.
\end{equation}
This construction defines a (linear) bijection $\phi$ between $V_{P,\mathbf Q}$ and $M^{\underline \mu}$. To conclude the proof we show that $\phi$ respects the $S_d\wr S_k$-action. 

Let us first consider the $S_d \wr S_k$-action on the element $\phi(x_{i_1,j_1} \cdots x_{i_t,j_t})$ shown in~\eqref{eq: basis vector from monomial}. Let $(\bm \alpha,\beta) \in S_d \wr S_k$. Let $(\id,\tilde b)$ be the coset representative for which $(\bm \alpha,\beta) (\id,b) \in (\id,\tilde b) \cdot S_d \wr S_{|\underline \mu|}$ and let $(\bm{\tilde \alpha},\tilde \beta)$ be such that
\begin{equation} \label{eq: new coset}
(\bm \alpha,\beta) (\id,b) = (\id,\tilde b) (\bm{\tilde \alpha}, \tilde \beta).
\end{equation}
Note that $\bm{\tilde \alpha}$, $\tilde \beta$ and $\tilde b$ are such that $\tilde \beta = \tilde b^{-1} \beta b \in S_{|\underline \mu|}$ and $\alpha_i = \tilde \alpha_{\tilde b^{-1}(i)}$ for all $i \in [k]$. In particular, $\tilde \alpha_i = \alpha_{\tilde b(i)}$ for all $i \in [k]$. Then we have 
\begin{align} \notag
\hspace{-0.5em}(\bm \alpha,\beta) \big((\id,b),(v_1,\ldots,v_k; u_1,\ldots,u_{\ell})\big) &= \big((\id,\tilde b),(\bm{ \tilde \alpha},\tilde \beta)(v_1,\ldots,v_k; u_1,\ldots,u_{\ell})\big)\\
&=\big((\id,\tilde b),(\tilde \alpha_1 v_{\tilde \beta^{-1}(1)},\ldots,\tilde \alpha_k v_{\tilde \beta^{-1}(k)}; \tilde \beta u_1,\ldots,\tilde \beta u_{\ell})\big) \notag \\
&=\big((\id,\tilde b),(\alpha_{\tilde b(1)} v_{\tilde \beta^{-1}(1)},\ldots, \alpha_{\tilde b(k)} v_{\tilde \beta^{-1}(k)}; \tilde \beta u_1,\ldots,\tilde \beta u_{\ell})\big).\label{eq: first application of (sigma,pi)}
\end{align}
On the other hand, for the corresponding element $x_{i_1,j_1} \cdots x_{i_t,j_t} \in V_{P,\mathbf Q}$ we have
\begin{align*}
(\bm \alpha,\beta)  x_{i_1,j_1} \cdots x_{i_t,j_t} =   x_{\alpha_{\beta(j_1)}(i_1),\beta(j_1)} \cdots x_{\alpha_{\beta(j_t)}(i_t),\beta(j_t)}.
\end{align*}
We will use the above identity to express $\phi((\bm \alpha, \beta)  x_{i_1,j_1} \cdots x_{i_t,j_t})$ in terms of $\phi(x_{i_1,j_1} \cdots x_{i_t,j_t})$. For $i \in [r]$ let $j \in [t-r]$ be the unique index such that $i \in I_j$, then $P_i$ is assigned basis 
\[
\beta(b(u_j( i))) = \tilde b(\tilde \beta(u_j( i))) = \tilde b ( \tilde u_j( i)),
\]
where we used the definition of $\tilde b,\bm{\tilde \alpha}, \tilde \beta$ from~\eqref{eq: new coset} and we let $\tilde u_a = \tilde \beta u_a$ for each $a \in [\ell]$. For each $i \in [k-r]$ we have that $\tilde v_{i+r}$ is the tabloid of shape $(d)$ with content $(1^d)$, and therefore $\tilde v_{i+r} = v_{\tilde \beta^{-1}(i+r)}$ where we use that $\tilde \beta \in S_{|\underline \mu|}$. For $j \in [t-r]$ and $i \in I_j$ we now consider the content of $\tilde v_{i}$. By definition it corresponds to the assignment of basis elements to the $j$ sets in 
\[
Q_{\tilde u_j^{-1}( i)} = Q_{(\tilde \beta u_j)^{-1}( i)} = Q_{u_j^{-1}(\tilde \beta^{-1}( i))}.
\]
It follows that the content of $\tilde v_{i}$ is the content of $v_{\tilde \beta^{-1}(i)}$. All variables corresponding to $Q_{\tilde u_j^{-1}(i)}$ are assigned basis $\tilde b \tilde u_j (\tilde u_j^{-1}(i)) = \tilde b(i)$ from which it follows that $\alpha_{\tilde b(i)}$ is applied to $v_{\tilde \beta^{-1}(i)}$. 
That is, 
\begin{align*}
\tilde v_{i} = \qquad \ytableausetup{boxsize=1.75em}\begin{ytableau}
\phantom{,}  & \cdots & \phantom{,} \\
w_{\tilde u_j^{-1}( i)}^1 \\
\vdots \\
w_{\tilde u_j^{-1}( i)}^{j} \\
\end{ytableau} &= \qquad \begin{ytableau}
\phantom{,}  & \cdots & \phantom{,} \\
\alpha_{\tilde b(i)} w_{u_j^{-1}( i)}^1 \\
\vdots \\
\alpha_{\tilde b(i)}w_{ u_j^{-1}( i)}^{j} \\
\end{ytableau}  = \alpha_{\tilde b(i)}  v_{{\tilde \beta}^{-1}(i)}.
\end{align*}
This shows, using \eqref{eq: first application of (sigma,pi)}, that 
\begin{align*}
\phi((\bm \alpha,\beta)  x_{i_1,j_1} \cdots x_{i_t,j_t}) &= \big((\id,\tilde b),(\alpha_{\tilde b(1)} v_{\tilde \beta^{-1}(1)},\ldots, \alpha_{\tilde b(k)} v_{\tilde \beta^{-1}(k)}; \tilde \beta u_1,\ldots,\tilde \beta u_{\ell})\big) \\
&= (\bm \alpha,\beta)   \phi(x_{i_1,j_1} \cdots x_{i_t,j_t}),
\end{align*}
which means that $\phi$ respects the action of $S_d \wr S_k$. 
\end{proof}

\section{Computational results} \label{sec: numerics}
Using the representative sets obtained in \cref{sec:reprset sk,sec:reprset wreath} we can symmetry reduce $\sdp(d,k,t)$. That is, we use the representative sets and \cref{prop: symmetry reduction} to obtain equivalent SDPs of smaller size. We furthermore use the $S_d \wr S_k$-invariance of feasible solutions, the MUB relations, and the tracial property to reduce the number of variables. We present computational results for (moderately) small $d,k,t$ in \cref{sec: tables}.\footnote{Recall, the code used is available at \url{https://github.com/MutuallyUnbiasedBases/}.}

\subsection{How to determine the value of~\texorpdfstring{$L$}{L} on monomials} \label{sec: L value}
An essential step in generating our SDPs is to decide which variables to use, i.e., for which $S_d \wr S_k$-orbits should we include a variable? For this we construct a list $\mathcal L$ of values $L(w)$ to use as variables as follows. Let~$\mathcal{R}=\{p_1,\ldots,p_m\}$ be a list of representatives for the $S_d \wr S_k$-orbits of monomials in~$\R\langle x_{i,j} \, | \, i \in [d], j \in [k]\rangle_{=2t}$. We order the monomials in the following way: we view $([d] \times [k])^t$ as $[k]^t \times [d]^t$ and then use the \emph{graded lexicographic order}, i.e., for two monomials $q,p$ we have $q \leq p$ if $\deg(q)<\deg(p)$ or if $\deg(q) = \deg(p)$ and $q$ precedes $p$ in the lexicographic order. This is a well ordering of the noncommutative monomials and any monomial is preceded by only finitely many other monomials. Then, as long as $\mathcal R$ is not empty, we remove the last monomial $p$ from $\mathcal R$ and express $L(p)$ as a linear combination of values~$L(q)$ using substitutions coming from the relations in $\Imub$ in the following order: 
\begin{enumerate}[nolistsep]
    \item If a sequence $x_{i,j} x_{i',j}$ occurs we use \ref{item:1}, i.e.~$x_{i,j} x_{i',j} = \delta_{i,i'} x_{i,j}$, to reduce the degree.
    \item If a base occurs once we use the $S_d \wr S_k$-invariance of $L$ and the POVM constraint \ref{item:2} to reduce the degree by one. 
    \item If a base occurs more than once but at least one of its elements occurs only once, we take the largest such basis element, say $x_{i,j}$, and we replace it by $I-\sum_{i' \neq i} x_{i',j}$. This gives $L(p) = L(p_0) - \sum_{i' \neq i} L(p_{i'})$ where $p_{i'}$ corresponds to replacing $x_{i,j}$ by the identity if $i'=0$ and by $x_{i',j}$ if $i>0$. Using the $S_d \wr S_k$-invariance of $L$ we have $L(p) = L(p_{i'})$ whenever $x_{i',j}$ does not occur elsewhere in $p$. This expresses $L(p)$ as a linear combination of ($L$ of) a monomial of lower degree $p_0$ and monomials $p_{i'}$ for which the number of basis elements occurring only once is reduced by one.
    \item If a sequence $x_{i,j} x_{i',j'} x_{i,j}$ occurs, we use \ref{item:3} to reduce the degree by two. 
    \item If a sequence $x_{i,j}u x_{i,j} v x_{i,j}$ occurs, we use \ref{item:4} if $x_{i,j}v x_{i,j} u x_{i,j}$ results in a monomial that is smaller in the lexicographic order or that can be reduced by \ref{item:3}. 
    \end{enumerate} 
    For each monomial $q$ we moreover use the tracial property of $L$, and the $S_d \wr S_k$-invariance of $L$ to reduce to a lexicographically smallest representative. If $q$ was not present in $\mathcal R$, we add it to $\mathcal R$. If none of the above substitutions can be applied to $p$, then we add~$L(p)$ to the list of variables $\mathcal L$. 
\begin{example} 
Let $d,k \geq 3$. The $S_d \wr S_k$-invariance and the POVM constraint  can be used to show
$L(x_{1,1}x_{1,2}x_{1,3}) = \tfrac{1}{d}\sum_{j=1}^d L(x_{j,1}x_{1,2}x_{1,3}) =\tfrac{1}{d} L(x_{1,2}x_{1,3}) = \tfrac{1}{d^2}$. 
In fact, one has to go to $t=3$ (degree $6$ monomials) to find a free variable, it corresponds to the monomial $(x_{1,1}x_{1,2}x_{1,3})^2$.
\end{example}

The above algorithm in fact does not capture all the linear relations between the values $L(w)$ for $w \in \ncx_{2t}$ that are implied by $\Imub$. Indeed, since we are not working with a Gr\"obner basis of $\Imub$, the above substitution method does not guarantee that $L=0$ on the entire ideal $\Imub$. Computing a Gr\"obner basis is in general a costly procedure. We therefore use the above procedure only to generate a list of variables for the semidefinite program. We then search for ``additional'' linear constraints between these variables by explicitly verifying whether $L=0$ on $\Imub$. That is, using the above procedure, we evaluate $L(qp)$ for, up to symmetry, all polynomials $q$ defining $\Imub$ and all monomials $p$ 
of degree at most $2t-\deg(q)$. If $L(qp)=0$ leads to a non-trivial constraint, we add it to the semidefinite program (if it is linearly independent of previously added constraints). In \cref{table: fullresults,table: i1results} below the resulting number of ``additional'' linear constraints is listed in the sixth column.

\subsection{Tables with computational results} \label{sec: tables}
In \cref{table: introtable} (in the introduction) we presented numerical results obtained by solving~$\sdp(d,k,t)$ for fixed~$d,k,t$ using the symmetry reduction and the above described method of determining the value of $L$ on monomials. The columns list the size (number of rows) $(dk)^{\lfloor t\rfloor}$ of the SDP before symmetry reduction, the number of variables after symmetry reduction, the number of ``additional'' linear constraints generated by the procedure in the previous section, and the sum and maximum of the block sizes after symmetry reduction. 
For the block sizes we moreover use the relations defining $\Imub$ to restrict to a subset of the pairs $(P,\bm Q)$ for which the moment matrix does not contain zero-rows (using relation \ref{item:1}), or linearly dependent rows (using relations \ref{item:1} or \ref{item:3}). The last column indicates the solution status as reported by SDPA-GMP, SDPA-DD or SDPA~\cite{Nakata10,YFFKNN12}. We remark that $\sdp(2,4,4.5)$ is infeasible already due to the ``additional'' linear constraints: they include $L(1)=0$. Below in \cref{table: fullresults} we show additional numerical results for dimension $d=6$ to illustrate the size of the involved SDPs. 
In Table~\ref{table: i1results} we present the same computational results for the relaxation of~$\sdp(d,k,t)$ obtained by working only with monomials in~$\R\langle x_{1,j} \, | \, j \in [k] \rangle_{2t}$.

\begin{table}[ht]\small
\begin{center}  
    \begin{tabular}{| r | r | r | r| r|r| r|r || l |}
    \hline
    $d$ & $k$ & $t$ &  $(dk)^{\lfloor t\rfloor}$&\#vars & \multicolumn{1}{r|}{$\#$linear} & 
    \multicolumn{2}{c||}{ block sizes} & result   \\ \cline{7-8}
     & & & & & constraints & sum & max &   \\\hline
6 & 4 & 4.5 & 331776 & 7 & 0 &  994 & 107 & feasible  \\
6 & 4 & 5 & 7962624 & 27 & 2 &  3697 & 319 & feasible \\
6 & 4 & 5.5 & 7962624 & 43 & 3 &  8049 & 577 & feasible \\\hline  
6 & 7 & 4.5 & 3111696 & 7 & 0 & 1627 & 146 & feasible \\
6 & 7 & 5 & 130691232 & 38 & 2  & 6749 & 389 & feasible \\\ 
6 & 7 & {$5.5$} &  130691232&   62 & 3& 18538 & 1107 & feasible    
\\\hline 
    \end{tabular} 
\end{center}
  \caption{An overview of results for the reduction of the full~$t$-th level moment matrix, using the symmetry group~$S_d \wr S_k$ (for level~$t$) or $S_{d-1} \times (S_d \wr S_{k-1}) $ (for level~$t+\frac12$). }\label{table: fullresults}
\end{table}

\begin{table}[ht]\small
\begin{center}  
    \begin{tabular}{| r | r | r | r| r|r| r|r || l |}
    \hline
    $d$ & $k$ & $t$ & $k^{\lfloor t\rfloor}$&\#vars & \multicolumn{1}{r|}{$\#$linear} &  
    \multicolumn{2}{c||}{ block sizes} & result   \\ \cline{7-8}
     & & & & & constraints & sum & max &   \\\hline
     2 & 4 & 4.5 & 256 & 5 & 0 & 48 & 24 & \textbf{infeasible}   \\
     3 & 5 & 4.5 & 625 & 5 & 0 & 95 & 32 &  \textbf{infeasible} \\
     4 & 6 & 5 & 7776 & 17 & 2 & 364 & 70 & \textbf{infeasible} \\ 
     5 & 7 & 6.5 & 117649 & 306 &49 & 3288 & 640 &  \textbf{infeasible} \\\hline  
     6 & 4 & 7.5 & 16384 &  19 & 5  & 586 & 293 & feasible\\    
     6 & 7 & 6.5 & 117649 & 306 & 49 & 3288 & 640 &  feasible \\
     6 & 8 & 6.5 & 262144 & 306 & 49  & 4096 & 707 &  feasible \\
    \hline\end{tabular} 
\end{center}
  \caption{An overview of the bounds for the submatrix of the~$t$-th level moment matrix indexed by monomials in~$\R\langle x_{1,j} \, | \, j \in [k]\rangle_{2t}$, using the symmetry group~$S_k$ (for level~$t$) or~$S_{k-1}$ (for level~$t+\frac12$).} \label{table: i1results}
\end{table}

\section*{Acknowledgements}

We thank Stefano Pironio and Erik Woodhead for interesting and useful discussions. In particular, we thank Stefano for pointing out that the substitution procedure in \cref{sec: L value} may miss linear constraints, and that infeasibility for~$(d,k)=(2,4)$ follows from these linear constraints. We thank Erik for independently checking some of our code used to reduce the number of variables. Sander Gribling was partially supported by the projects QUDATA (ANR-18-CE47-0010) and QOPT (QuantERA ERA-NET Cofund 2022-25). Most of this research was carried out while Sander Gribling was affiliated with Universit\'e Paris Cit\'e, IRIF, and Sven Polak was affiliated with Centrum Wiskunde \& Informatica, Amsterdam.

\bibliographystyle{alphaurl}
\bibliography{MUBbib.bib}

\clearpage 
\appendix 
\section{\texorpdfstring{$C^*$}{C*}-algebraic formulation of the MUB problem} \label{appA}
The theorem below was proven in~\cite{NPA12}, we give a slightly simpler proof (inspired by their proof). 
\thrmCalg*
\begin{proof}
Let $\{u_{i,1}\}_{i \in [d]}, \{u_{i,2}\}_{i \in [d]}, \ldots, \{u_{i,k}\}_{i \in [d]}$ be  ONB's of $\C^d$. Let $X_{i,j} = u_{i,j} u_{i,j}^*$ be the rank-$1$ projector associated to $u_{i,j}$ for all $i \in [d],j \in [k]$. Then the $C^*$-algebra $\A = \C\langle \mathbf X\rangle$ generated by the $X_{i,j}$, and the Hermitian operators $X_{i,j}$, satisfy the relations \labelcref{item:1,item:2,item:3,item:4}. 

Conversely, suppose there exist $\A$ and Hermitian $X_{i,j} \in \A$ satisfying the relations \labelcref{item:1,item:2,item:3,item:4}. We show how to construct $k$ mutually unbiased bases in dimension $d$. For all $i \in [d],j \in [k]$, let $Z_{ij} \in M_d(\mathcal A)$ 
be defined as\footnote{The choice for $1$ and $2$ here is arbitrary. All that matters is they are fixed and different.} 
\[
Z_{i,j} := d \Big[ X_{1,2} X_{a,1} X_{i,j} X_{b,1} X_{1,2} \Big]_{a,b \in [d]}.
\]
Clearly the matrices $Z_{i,j}$ are Hermitian. We now verify that the matrices $Z_{ij}$ satisfy the relations~\ref{item:1} and~\ref{item:3}. Let $i,i' \in [d], j,j' \in [k]$, we have 
\begin{align*}
(Z_{i,j}Z_{i',j'})_{a,b} &= d^2 \sum_{c \in [d]} X_{1,2} X_{a,1} X_{i,j} X_{c,1} X_{1,2} X_{1,2} X_{c,1} X_{i',j'} X_{b,1} X_{1,2} \\
&\overset{\ref{item:1},\ref{item:3}}{=} d^2\cdot \frac{1}{d} \sum_{c \in [d]} X_{1,2} X_{a,1} X_{i,j} X_{c,1} X_{i',j'} X_{b,1} X_{1,2} \\
&\overset{\ref{item:2}}{=} d X_{1,2} X_{a,1} X_{i,j} X_{i',j'} X_{b,1} X_{1,2}.
\end{align*}
When $j = j'$ the above shows, together with~\ref{item:1} for the $X_{i,j}$'s, that $Z_{i,j}Z_{i', j} = \delta_{i,i'} Z_{i,j}$. 
A similar computation shows that when $j \neq j'$ we have 
$Z_{i,j}Z_{i',j'}Z_{i,j} = \frac{1}{d} Z_{i,j}$.
Finally, using relation~\ref{item:2} we see that 
    $\sum_{a,b \in [d]} \sum_{i \in [d]} (Z_{i,j})_{a,b} = d X_{1,2}$.
It follows from~\ref{item:2} and~\ref{item:3} for the $X_{i,j}$'s that $X_{1,2} \neq 0$. Hence $\sum_{i \in [d]} Z_{i,j} \neq 0$, and since~\ref{item:3} holds for the $Z_{i,j}$, we therefore have that $Z_{i,j} \neq 0$ for all $i,j$. 

Now observe that $\C\langle X_{1,2} X_{a,1} X_{i,j} X_{b,1} X_{1,2}: a,b,i \in [d], j \in [k]\rangle$ is a commutative $C^*$-algebra because of~\ref{item:4}. We will restrict our attention to this sub-algebra and refer to it as $\mathcal A$. A commutative $C^*$-algebra can be viewed as an algebra of continuous functions in the following sense, see, e.g.,~\cite[II.2.2.4]{Blackadar06}. Using the Gelfand representation, any commutative $C^*$-algebra $\mathcal A$ is $*$-isomorphic to $C_0(\hat{\mathcal A})$, where $\hat{\mathcal A}$ is the set of non-zero $*$-homomorphisms from $\mathcal A$ to $\C$ (i.e., the set of characters), and $C_0(\hat{\mathcal A})$ is the set of continuous functions on $\hat{\mathcal A}$ that vanish at infinity ($\hat{\mathcal A}$ is a locally compact Hausdorff space, using the weak-$*$ topology).
We may thus write 
\[
Z_{i,j} := d \Big[ \Theta_{a,b}^{i,j}(\hat a) 
\Big]_{a,b \in [d]}
\]
where $\Theta_{a,b}^{i,j} \in C_0(\hat{\mathcal A})$ and $\hat a \in \hat{\mathcal A}$. 

Now for each $\hat a \in \hat{\mathcal A}$ define $Z_{i,j}^{\hat a} := d [ \Theta_{a,b}^{i,j}(\hat a)]_{a,b \in [d]}$.
Since the multiplication in $C_0(\hat{\mathcal A})$ is pointwise, the relations \ref{item:2} and~\ref{item:3} also hold for the matrices $Z_{i,j}^{\hat a}$ (for $i \in [d],j \in [k]$). Fix $i' \in [d], j' \in [k]$. Since $Z_{i',j'} \neq 0$ there exists an $\hat a$ such that $Z_{i',j'}^{\hat a} \neq 0$. Since~\ref{item:3} holds for the matrices $Z_{i,j}^{\hat a}$ ($i\in[d], j \in [k]$), we know that all of them are non-zero. Together with~\ref{item:1} this implies that each $Z_{i,j}^{\hat a}$ has rank equal to $1$. Thus for each $j \in [k]$, the projectors $Z_{i,j}^{\hat a}$ ($i \in [d]$) define an orthonormal basis. Relation~\ref{item:3} for the $Z_{i,j}$'s implies that these bases are mutually unbiased. 
\end{proof}

\end{document}